\documentclass[a4paper,11pt]{amsart}

\usepackage{amsmath,amssymb}
\usepackage{amsthm}
\usepackage{mathrsfs}
\usepackage{enumitem}
\usepackage{mathtools}
\usepackage[all,cmtip,2cell]{xy}
\usepackage{array}
\theoremstyle{plain}
\newtheorem{thm}{Theorem}[section]
\newtheorem{prop}[thm]{Proposition}
\newtheorem{lem}[thm]{Lemma}

\newtheorem*{thm*}{Theorem}

\theoremstyle{definition}
\newtheorem{defi}[thm]{Definition}

\newtheorem*{NaC}{Notation and Convention}
\newtheorem*{ACK}{Acknowledgement}

\theoremstyle{remark}

\numberwithin{equation}{subsection}

\newcommand{\Z}{\mathbb{Z}}
\newcommand{\Q}{\mathbb{Q}}

\newcommand{\C}{\mathbb{C}}

\renewcommand{\P}{\mathbb{P}}
\newcommand{\F}{\mathbb{F}}

\renewcommand{\a}{\alpha}
\renewcommand{\b}{\beta}

\newcommand{\D}{\Delta}
\newcommand{\e}{\varepsilon}
\newcommand{\g}{\gamma}
\newcommand{\G}{\Gamma}

\renewcommand{\L}{\Lambda}
\newcommand{\s}{\sigma}
\renewcommand{\t}{\tau}
\newcommand{\vp}{\varphi}

\newcommand{\mc}{\mathcal}

\newcommand{\tb}{\textbf}

\newcommand{\bbox}{\hfill \rule[-2pt]{5pt}{10pt}}

\newcommand{\ol}{\overline}

\newcommand{\wt}{\widetilde}

\renewcommand{\sb}{\subset}
\renewcommand{\sp}{\supset}

\newcommand{\hra}{\hookrightarrow}

\newcommand{\dra}{\dashrightarrow}

\DeclareMathOperator{\rk}{rk}
\DeclareMathOperator{\id}{id}

\DeclareMathOperator{\Sym}{\mathrm{Sym}}

\DeclareMathOperator{\Proj}{Proj}

\DeclareMathOperator{\NS}{NS}
\DeclareMathOperator{\Exc}{Exc}

\DeclareMathOperator{\Pic}{\mathrm{Pic}}

\DeclareMathOperator{\eu}{\chi_{\mathrm{Top}}}

\DeclareMathOperator{\Bl}{\mathrm{Bl}}

\DeclareMathOperator{\codim}{\mathrm{codim}}

\DeclareMathOperator{\Bs}{\mathrm{Bs}}

\DeclareMathOperator{\pr}{pr}

\newcommand{\bhline}[1]{\noalign{\hrule height #1}} 

\title[{\tiny On almost Fano 3-folds with del Pezzo fibrations.}]{On the existence of almost Fano threefolds with del Pezzo fibrations}
\author[T. FUKUOKA]{Takeru Fukuoka}
\date{\today}
\address{Graduate School of Mathematical Sciences\\The University of Tokyo\\3-8-1 Komaba\\Meguro-ku, Tokyo 153-8914, Japan}
\email{tfukuoka@ms.u-tokyo.ac.jp}
\begin{document}
\maketitle
\begin{abstract}
By Jahnke-Peternell-Radloff and Takeuchi, almost Fano threefolds with del Pezzo fibrations were classified. 
Among them, there exists 10 classes such that the existence of members of these was not proved. 
In this paper, we construct such examples belonging to each of 10 classes.
\end{abstract}
\tableofcontents
\setcounter{section}{0}
\section{Introduction}
Throughout this paper, we work over the field of complex numbers $\C$. 
Let $X$ be a non-singular projective three dimensional variety. 
If $-K_{X}$ is ample, $X$ is called a \textit{Fano} threefold. 
As a weaker condition than Fano, we consider a non-singular projective threefold with a nef and big anti-canonical divisor. 
This variety is called a \textit{weak Fano} threefold.  
If $-K_{X}$ is not ample but nef and big, we call $X$ an \textit{almost Fano} threefold. 

The classification of Fano threefolds is very important and many people studied it. 
Roughly speaking, we may regard almost Fano threefolds as ``degenerated'' Fano threefolds. 
In particular, the classification of almost Fano threefolds is related to that of singular Fano threefolds. 
For example, some Gorenstein terminal Fano threefolds have almost Fano threefolds as their small resolutions. 
Therefore, the classification of almost Fano threefolds is also an important problem. 

To classify (almost) Fano threefolds, 
we are interested in possible tuples of values of some invariants. 
As examples of invariants, we consider
the \textit{Picard rank} $\rho(X):=\dim_{\Q} (\NS(X) \otimes_{\Z}\Q)$, 
the \textit{anti-canonical degree} $(-K_{X})^{3}$, 
the \textit{Hodge number} $h^{1,2}(X):=\dim H^{2}(X,\Omega_{X})$, %=b_{3}(X)/2
and types of contractions of extremal rays of $X$. 
For each fixed tuple of values, 
%of invariants 
we call the set 
\[\{X \mid \text{ the tuple of invariants of } X \text{ is equal to the fixed one }\} \]
the \textit{class} corresponding to the fixed tuple. % of values of invariants. 
To classify (almost) Fano threefolds, we need to reveal whether each class has a member or not. 

The classification of Fano threefolds is given by the papers : 
%\cite{KO73}, 
\cite{Isk77,Isk78}, 
\cite{Fuj80,Fuj81,Fuj84}, \cite{Sho79a}, \cite{Sho79b},\cite{Take89}, \cite{Muk95}, \cite{MM81,MM86,MM03}, and so on. 
In particular, they classified the possible tuples of values of some invariants and proved that each class has a member. 
\cite{Fanobook} is a survey of the classification of Fano varieties including these results. 

The classification of almost Fano threefolds $X$ with the minimum Picard rank, that is, $\rho(X) = 2$ has been most intensively studied. Here we recall some known results. 
By virtue of Mori theory, $X$ has the contraction of the $K_{X}$-negative ray $\vp \colon X \to W$ and the contraction of the $K_{X}$-trivial ray $\psi \colon X \to \ol{X}
%\simeq \Proj \bigoplus_{n \geq 0} H^{0}(X,-nK_{X})
$. 
Then $\psi$ contracts a divisor or finitely many curves. 
In the latter case, let $\chi \colon X \dra X^{+}$ denote the flop of $\psi$ \cite{Kol89}. 
Thus we have two cases as the following diagrams. 
\[\xymatrix{
\text{Case (A)}&&X \ar[rd]^{\psi\text{ : divisorial }} \ar[ld]_{\vp}& \\
&W&&\ol{X}.
}\]
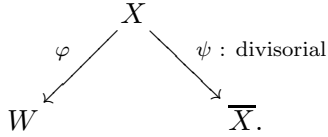
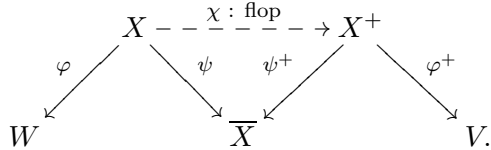
\[\xymatrix{
\text{Case (B)}&&X \ar[rd]^{\psi} \ar[ld]_{\vp} \ar@{-->}[rr]^{\chi \text{ : flop }}&&X^{+}\ar[rd]^{\vp^{+}} \ar[ld]_{\psi^{+}}& \\
&W&&\ol{X}&&V.
}\] 
\begin{itemize}
\item Jahnke-Peternell-Radloff treated Case (A) in \cite{JPR05}. 
They narrowed down tuples of values of invariants. 
The existence of a member of each class was proved except for the two classes
appearing in TABLE \ref{table:A} below. 
\item In \cite{JPR11} (resp. \cite{CM13}), Jahnke-Peternell-Radloff (resp. Cutrone-Marshburn) treated Case (B) when $\vp$ or $\vp^{+}$ is not a divisorial contraction (resp. $\vp$ and $\vp^{+}$ are divisorial contraction ). 
They narrowed down tuples of values of invariants. The existence of a member of each class was proved except for some classes. 
\item In \cite{Take09}, Takeuchi treated Case (B) when $\vp$ is a del Pezzo fibration of degree $d \neq 6$. 
He classified tuples of values of invariants and prove the existence of members of each class. Takeuchi's works and Jahnke-Peternell-Radloff's works are independent each other. 
\item In \cite{Vol01}, Vologodsky treated Case (B) when $\vp$ and $\vp^{+}$ are del Pezzo fibrations. 
He classified the possible values of $(-K_{X})^{3}$. 
Unfortunately, due to mistakes in \cite[Proposition 2.2]{Vol01}, 
there are some missing values in his classification. 
\cite{Take09}, \cite{JPR11} and this paper fill the whole missing values of anti-canonical degrees by constructing examples. 
\end{itemize}

In this paper, we mainly consider the case where $\vp$ is a del Pezzo fibration. % (see \hyperlink{NC}{Notation and Convention}). 
In this case, it is known that $W=\P^{1}$. 
By summarizing the above known results, there exists 10 classes such that it is yet to be known whether these have members or not. 
Our main theorem is to show the existence of members of each class. 
\begin{thm}\label{Mainthm}
Each class belonging to the following TABLE \ref{table:A} and \ref{table:B} has a member. 
In particular, there exists examples of each almost Fano threefold with a del Pezzo fibration that appears in \cite{JPR05} and \cite{JPR11}. 
\end{thm}
\begin{table}[h]
\caption{\tb{Case (A)}}
\begin{tabular}{|c||c|c|c|c|c|c|}
\hline
Name&$\vp$ : $dP_{d}$& Type of $\psi \colon X \to \ol{X}$ &$(-K_{X})^{3}$& $D_{X}=\Exc(\psi)$ & $h^{1,2}(X)$ & $\exists$ \\
\hline \hline
(A-1)&$6$&$(g,d)=(1,6)$&$12$ & $(-K_{X})-F$ &``2''& $\S$ \ref{ss:A1} \\ 
\hline
(A-2)&$5$&$(g,d)=(1,5)$&$10$ & $(-K_{X})-2F$ &``6''& $\S$ \ref{ss:A2} \\ 
\hline
\end{tabular}
\label{table:A}
\caption{\tb{Case (B)}}
\begin{tabular}{|c||c|c|c|c|c|c|c|c|}
\hline
Name&$\vp$ : $dP_{d}$&$V$& $\vp^{+} \colon X^{+} \to V$ &$(-K_{X})^{3}$& $D_{X}$ & $h^{1,2}(X)$ & $\exists$ \\
 \hline \hline
(B-i-1)&6&$B(4)$& $(g,d)=(1,6)$ & $8$ & $3(-K_{X})-2F$ & 3 & $\S$ \ref{ss:Bi1} \\ 
\hline
(B-i-2)&6&$V(10)$& $(g,d)=(1,6)$ & $6$ & $2(-K_{X})-F$& 3 & $\S$ \ref{ss:Bi2}\\ 
\hline
(B-i-3)&6&$V(9)$& $(g,d)=(1,6)$ & $4$ & $3(-K_{X})-F$ & 4 & $\S$ \ref{ss:Bi3} \\ 
\bhline{2pt}
(B-ii)&6& $\P^{2}$&$\deg \text{(disc.)}=4$ & $14$ & $(-K_{X})-F$ & ``2'' & $\S$ \ref{ss:Bii} \\
\bhline{2pt}
(B-iii-1)& 6&$\P^{1}$&$dP_{6}$ & $12$ & $(-K_{X})-F$ & ``2'' & $\S$ \ref{ss:Biii1} \\
\hline
(B-iii-2)&6& $\P^{1}$&$dP_{6}$ & $6$ & $2(-K_{X})-F$ & ``4'' & $\S$ \ref{ss:Biii2} \\
\hline
(B-iii-3)&6&$\P^{1}$&$dP_{6}$ & $4$ & $3(-K_{X})-F$ & ``3'' & $\S$ \ref{ss:Biii3} \\
\hline
(B-iii-4)&6&$\P^{1}$&$dP_{6}$ & $2$ & $6(-K_{X})-F$ & ``5'' & $\S$ \ref{ss:Biii4} \\
\hline
\end{tabular}
\label{table:B}
\end{table}

\begin{center}
\tb{Notation for TABLE \ref{table:A}.}
\end{center}
\begin{itemize}
\item The second row from the left denotes the degree of the del Pezzo fibration $\vp \colon X \to \P^{1}$.% (see \hyperlink{NC}{Notation and Convention}). 
\item The third row from the left denotes the types of $\psi$. 
``$(g,d)$'' means that $\psi$ is the blowing-up along a non-singular curve $C$ of genus $g$ with $(-K_{\ol{X}}).C=d$. 
\item The fifth row from the left denotes the description of $D_{X}:=\Exc(\psi)$ in $\Pic(X)=\Z \cdot [-K_{X}] \oplus \Z \cdot [F]$. Here, $F$ is a general $\vp$-fiber. 
\item The second rows from the right denote the Hodge numbers $h^{1,2}(X)$ of members of a class. ``$n$'' means that there exists at least one member $X$ of the class such that $h^{1,2}(X)=n$. 
%We compute the Hodge numbers $h^{1,2}(X)$ in $\S$ \ref{ss:hodge}. 
\item The rightmost rows denote subsections including proofs of the existence of a member. 
\end{itemize}
\begin{center}
\tb{Notation for TABLE \ref{table:B}.}
\end{center}
\begin{itemize}
\item The second row from the left denotes the degree of the del Pezzo fibration $\vp \colon X \to \P^{1}$. 
\item The third row from the left denotes the types of $V$. 
$B(m)$ (resp. $V(g)$) denotes a del Pezzo (resp. Mukai) threefold of degree $m$ (resp. genus $g$). 
%More precisely, see \hyperlink{NC}{Notation and Convention}. 
\item The fourth row from the left denotes the types of $\vp^{+}$. 
\begin{itemize}
\item ``$(g,d)$'' means that $\vp^{+}$ is the blowing-up along a non-singular curve $C$ of genus $g$ with $-K_{V}.C=i_{V} \cdot d$. 
Here, $i_{V}$ denotes the Fano index of $V$. 
%(see \hyperlink{NC}{Notation and Convention}). 
\item ``deg(disc.)=$4$'' means that $\vp^{+} \colon X^{+} \to V$ is a conic bundle and the degree of the discriminant divisor is 4. Note that we obtain $V=\P^{2}$ in this case. 
\item ``$dP_{6}$'' means that $\vp^{+} \colon X^{+} \to V$ is a del Pezzo fibration of degree $6$. In this case, we obtain $V=\P^{1}$. 
\end{itemize}
\item The fifth row from the left denotes the description of the divisor $D_{X}:=\chi^{-1}_{\ast}D_{X^{+}}$ in $\Pic(X)=\Z \cdot [-K_{X}] \oplus \Z \cdot[F]$. Here, $F$ is a general $\vp$-fiber and $D_{X^{+}}:=\Exc(\vp^{+})$ (resp. $\vp^{+\ast}\mc O_{\P^{2}}(1)$, $\vp^{+\ast}\mc O_{\P^{1}}(1)$) in Case (B-i) (resp. (B-ii),(B-iii)). 
\item The second rows from the right denote the Hodge numbers $h^{1,2}(X)$ of members of a class. ``$n$'' means that there exists at least one member $X$ of the class such that $h^{1,2}(X)=n$. 
The Hodge numbers $h^{1,2}(X)$ of an arbitrary member of classes in Case (B-i) are given in $\S$ \ref{ss:hodge}. 
\item The rightmost rows denote subsections including proofs of the existence of a member. 
\end{itemize}

We explain a sketch of proof of Theorem \ref{Mainthm}. 

In Case (B-i), the construction of a member of the class is reduced to the construction of a Fano threefold containing an elliptic curve of degree 6. 
In Case (B-i-1), we can construct a member by standard arguments. 
In Case (B-i-2) and (B-i-3), we use theory of N\'eron-Severi lattices of K3 surfaces. 
This strategy is the same one in \cite{CM13}. 
%(Theorem \ref{thm:NSL} (1) = \cite[Lemma~2.9]{Morr84}) 

%A lattice, which is needed to use Theorem \ref{thm:NSL}, can be make by counting backward from some numerical data given by TABLE \ref{table:B}. We omit this process and give a lattice directly. 

In Case (A-1), Case (B-ii) and Case (B-iii), 
the idea of our construction of a member of those classes based on an elementary birational transformation as follows. 
Let $\Q^{2} \sb \P^{3}$ be a smooth quadric surface and take general three points $p_{1},p_{2},p_{3}$ on $\Q^{2}$. 
Then the linear span of three points $p_{1},p_{2},p_{3}$ is a plane and the intersection of the plane and $\Q^{2}$ is a conic $C$. 
Let $\s \colon F \to \Q^{2}$ be the blowing-up of $\Q^{2}$ at $p_{1},p_{2},p_{3}$ and $\wt{C}$ the proper transform of $C$. 
Then $\wt{C}$ is $(-1)$-curve and hence we obtain the blowing down $\t \colon F \to S$ of $\wt{C}$. Note that $S$ is a del Pezzo surface of degree 6. The following proposition is a relativization of this birational transformation $\Q^{2} \gets F \to S$. 
See Proposition \ref{prop:1} for precise statement. 
\begin{prop}
Let $\pi \colon W \to \P^{1}$ be a quadric fibration 
%(see \hyperlink{NC}{Notation and Convention}) 
and $B \sb W$ a smooth curve and 
$\t \colon Z:=\Bl_{B}W \to W$ the blowing-up along $B$. 
We assume the following condition for a pair $(\pi \colon W \to \P^{1},B)$.
\[
\left\{ \begin{array}{ll}
\deg (\pi|_{B} \colon B \to \P^{1})=3 \text{ and } \\
-K_{Z} \text{ is $p$-nef and $p$-big with } p := \pi \circ \t \colon Z \to \P^{1}. 
\end{array} \right. \]
Then there exists a birational map $\Phi \colon Z \dra Y$ over $\P^{1}$ and a birational morphism $\mu \colon Y \to X$ over $\P^{1}$. Moreover, the following holds.
\begin{itemize}
\item $\Phi$ is isomorphic in codimension 1. 
\item $\mu$ is the blowing-up along a $\vp$-section $C$. Here, $\vp \colon X \to \P^{1}$ is a structure morphism onto $\P^{1}$. 
\item $\vp \colon X \to \P^{1}$ is a del Pezzo fibration of degree $6$. 
\item $(-K_{X})^{3}=\frac{3(-K_{W})^{3}-16g_{B}-32}{4}$ and $-K_{X}.C=\frac{8 (-K_{W}).B-24g_{B}-(-K_{W})^{3}-32}{8}$. 
\end{itemize}
\end{prop}
This proposition is a variant of Takeuchi's 2-ray game \cite{Take89}. 
By using this proposition, the construction problem can be reduced to the construction of a pair $(\pi \colon W \to \P^{1},B)$ satisfying the condition above. 

As an example, assume that $W$ is the blowing-up of a quadric threefold $\Q^{3}$ along a conic $\G$ (resp. the blowing-up of $\P^{3}$ along an elliptic curve $\G$ of degree $4$) . 
Then the construction of $(\pi \colon W \to \P^{1},B)$ is reduced to the construction of a pair of curve $(B,\G)$ in $\Q^{3}$ (resp. $\P^{3}$), where $\G$ is a conic (resp. elliptic curve of degree $4$) and $B$ satisfies $\deg B=\#(B \cap \G)+3$ (resp. $2\deg B=\#(B \cap \G)+3$) . 
We construct such pair by using a theory of N\'eron-Severi lattices of K3 surfaces except for Case (B-iii-1) and Case (A-1). 
Actually, our constructions of a member of each classes use the quadric fibrations above except for Case (B-iii-3). 
In Case (B-iii-3), we use a $(2,2)$ divisor in $\P^{3} \times \P^{1}$ for $W$. 

%Making a lattice, which is needed to use Theorem \ref{thm:NSL}, needs a kind of guessworks and some working hypothesizes based on numerical conditions, which comes from TABLE \ref{table:B} and Proposition \ref{prop:1}. We omit this process and give a lattice directly. 

In Case (A-2), our construction of a member is similarly based on the birational transformation of surfaces as follows.
We take general 5 points $p_{1},\ldots,p_{5}$ in $\P^{2}$ and consider the conic $C$ passing through $p_{1},\ldots,p_{5}$. Let $\s \colon F \to \P^{2}$ be the blowing-up at $p_{1},\ldots,p_{5}$ and $\wt{C}$ be the proper transform of $C$. 
Then $\wt{C}$ is a $(-1)$-curve and hence we obtain the blowing-down $\t \colon F \to S$ of $\wt{C}$ from $F$ onto a del Pezzo surface $S$ of degree $5$. 
We prove a relativization of this birational transformation (see Proposition \ref{prop:2}) and use it to construct an example belonging to Case (A-2). 
%In particular, a Torelli-type theorem for N\'eron-Severi lattice (\cite[Lem 2.9]{Morr84} ) and the characterization of base-point-freeness and very ampleness of line bundles (\ref{SD74} or \ref{Rei88} ) are useful to our goal. 
\begin{NaC}
We basically adopt the terminology of \cite{HarBook} and \cite{KMBook}. 
\begin{itemize}
\item For a closed subvariety $Y \sb X$, $N_{Y}X$ denotes the normal sheaf. 
\item For a birational morphism $\vp \colon X \to Y$ with relative Picard rank one, $\Exc(\vp)$ denotes the exceptional set of $\vp$. 
If $\codim_{X}\Exc(\vp)=1$, we call $\vp$ a \textit{divisorial} contraction. 
If $\codim_{X}\Exc(\vp) \geq 2$, we call $\vp$ a \textit{small} contraction. 
Assume that $X$ is non-singular projective variety. 
If $\vp$ is a divisorial contraction and $K_{X} \sim_{\vp} 0$, we call $\vp$ a \textit{crepant} contraction. 
If $\vp$ is a small contraction and $K_{X} \sim_{\vp} 0$, we call $\vp$ a \textit{flopping} contraction. 
%In this paper, we treat only non-singular Mori fiber spaces. 
\item A \textit{Mori fiber space} is a contraction $\vp \colon X \to S$ of a $K_{X}$-negative ray with non-singular projective variety $X$ and normal projective variety $S$ such that $\dim S<\dim X$. 
\item A \textit{del Pezzo fibration} is a Mori fiber space $\vp \colon X \to S$ with $\dim X=3$ and $\dim S=1$. The \textit{degree} of del Pezzo fibration $\vp \colon X \to S$ is a anti-canonical degree $(-K_{F})^{2}$ for a general $\vp$-fiber $F$. A del Pezzo fibration of degree 8 is called a \textit{quadric fibration}. 
\item A \textit{conic bundle} is a Mori fiber space $\vp \colon X \to S$ with $\dim X=3$ and $\dim S=2$. The \textit{discriminant divisor} $\D$ of conic bundle $\vp \colon X \to S$ is a divisor of $S$ given by $\D=\{x \in S \mid \vp^{-1}(x) \text{ is singular } \}$. Note that a conic bundle $\vp \colon X \to S$ is always flat and hence $S$ is non-singular. 
\item For a non-singular Fano variety $V$, $i_{V}:=\max\{ i \in \Z_{>0} \mid -K_{V}=i \cdot H \text{ for some Cartier divisor } H\}$ denotes the \textit{Fano index} of $V$. 
\item $\Q^{n}$ denotes a $n$-dimensional smooth quadric hypersurface. 
\item $B(m)$ is a \textit{del Pezzo threefold} of degree $m$, which means $V=B(m)$ is a Fano threefold with $i_{V}=2$ and $(-K_{V})^{3}=8m$.
\item $V(g)$ is a \textit{Mukai threefold} of genus $g$, which means $V=V(g)$ is a Fano threefold with $i_{V}=1$ and $(-K_{V})^{3}=2g-2$. 
\item For a locally free sheaf $\mc E$, we set $\P(\mc E):=\Proj \Sym^{\bullet}\mc E$ and $\mc O_{\P(\mc E)}(1)$ as the tautological bundle of $\P(\mc E)$. 
\item The $\P^{1}$-bundle $\P_{\P^{1}}(\mc O \oplus \mc O(n))$ over $\P^{1}$, a Hirzebruch surface, is denoted by $\F_{n}$. $\mc O_{\F_{n}}(1)$ denotes the tautological bundle of $\F_{n}=\P_{\P^{1}}(\mc O \oplus \mc O(n))$. 
\item We say that $\vp \colon X \to S$ is a \textit{$\P^{2}$-bundle} if $\vp$ is a projection morphism of a projective space bundle $\P(\mc E) \to S$ associated to some locally free sheaf $\mc E$ of rank 3 over $S$. 
%Note that a del Pezzo fibration of degree 9 is always $\P^{2}$-bundle. 
\item For a birational map $f \colon X \dra Y$ and a closed subscheme $S \sb X$, $f_{\ast}S$ and $S_{Y}$ denotes the proper transform of $S$. 
\end{itemize}
\end{NaC}
\begin{ACK}
I am deeply grateful to Professor Hiromichi Takagi, my supervisor, for his valuable comments, suggestions and encouragement. 
I am also grateful to Professor Kiyohiko Takeuchi who showed me his unpublished works about the classification of weak Fano threefolds with Picard rank two. 
%I thank Genki Sato, Akihiro Kanemitsu, Fumiaki Suzuki, Hokuto Konno who pointed out and correct my errors in the English language of this paper. 
This paper is based on my master thesis at the Graduate School of Mathematical Sciences, The University of Tokyo. 
\end{ACK}
\section{Preliminaries}
\subsection{Divisors on K3 surfaces}\label{subsec:K3}
In this subsection, we review some theory for K3 surfaces. 
First, we recall a result of a theory of N\'eron-Severi lattices of K3 surfaces. 
We also recall the fact that the fundamental domain, for the Picard-Lefschetz reflection on the positive cone of a K3 surface, is the closure of the K\"ahler cone. 
We collect these results as follows. 
\begin{thm}[{\cite[Lemma~2.9]{Morr84}, \cite[Proposition~3.10]{BHPV}}]\label{thm:NSL}
The following hold.
\begin{enumerate}
\item Let $\rho$ be an integer with $1 \leq \rho \leq 10$ and $\L$ an even lattice with signature $(1,\rho-1)$. Then there exists a projective K3 surface $S$ and an isometry $i \colon \L \to \Pic(S)$. Here, the lattice structure of $\Pic(S)$ is given by whose intersection form. 
\item Let $H$ be a line bundle of a K3 surface $S$. If $H^{2}>0$, then there exists an isometry $i \colon \Pic(S) \to \Pic(S)$ such that $i(H)$ is nef and big. 
\end{enumerate}
\end{thm}
We review some characterizations of very ampleness or base point freeness of nef line bundles of K3 surfaces. This subject is mainly treated by \cite{SD74}. 
\begin{thm}[{\cite{SD74},\cite{Rei88}}]\label{thm:SDR}
Let $L$ be a nef line bundle on a K3 surface $S$. Then the following hold. 
\begin{enumerate}
\item The following are equivalent. 
\begin{enumerate}
\item $L$ is globally generated. 
\item $|L|$ is fixed part free. 
\item There is no line bundle $D \in \Pic(S)$ such that $D^{2}=0$ and $L.D=1$. 
\end{enumerate}
\item Assume $L^{2} \geq 4$. Then the following are equivalent. 
\begin{enumerate}
\item $L$ is very ample. 
\item There is no line bundle $D \in \Pic(S)$ satisfying one of the following.
\begin{enumerate}
\item $D^{2}=-2$ and $L.D=0$.
\item $D^{2}=0$ and $L.D=1$ or $2$.
\item $D^{2}=2$ and $L \equiv 2D$. 
\end{enumerate}
\end{enumerate}
\item Assume $L$ is base point free. If $L^{2} > 0$, then there exists a non-singular curve $C$ of $|L|$. If $L^{2}=0$, then there exists an elliptic curve $C$ such that $nC \in |L|$ for some $n>0$. In particular, if $L^{2}=0$ and $L$ is a generator of the lattice $\Pic(S)$, there exists an elliptic curve $C$ such that $C \in |L|$.
\end{enumerate}
\end{thm}
%\begin{proof}
%For (1) and (2), see \cite{SD74}. 
%One can prove (1) and (2) as a consequence of Reider's theorem \cite{Rei88}. 
%For (3) See \cite[Proposition~2.6]{SD74}. 
%\end{proof}
We prepare two easy lemmas to prove Theorem \ref{Mainthm}. 
\begin{lem}\label{lem:1}
Let $L$ be an effective divisor on a K3 surface $S$ and $N$ be the fixed part of $|L|$. Then every irreducible component of $N$ is a $(-2)$-curve. 
Moreover, if $L^{2} \geq 0$, then $L-N \neq 0$. 
\end{lem}
\begin{proof}
Let $C$ be a irreducible component of $N$. 
Then the multiplication map 
$H^{0}(S,\mc O(L-C)) \otimes H^{0}(S,\mc O(C)) \to H^{0}(S,\mc O(L))$ 
is bijective and hence $h^{0}(S,\mc O_{S}(C))=1$. 
By the Serre duality and the Riemann-Roch theorem, 
we have $C^{2}<0$ and hence $C$ is $(-2)$-curve. 
If $L^{2} \geq 0$, we have $h^{0}(L) \geq 2$ by the Riemann-Roch theorem and hence $L \neq N$. 
\end{proof}
\begin{lem}\label{lem:2}
Let $S$ be a K3 surface and $H$ a very ample divisor with $H^{2}=6$. 
Then $S$ is a complete intersection of a quadric hypersurface $Q$ and a cubic hypersurface in $\P^{4}$. 
Moreover, $Q$ is smooth if there exists no effective divisors $C_{1},C_{2}$ such that 
$H=C_{1}+C_{2}$ , $C_{1}^{2}=C_{2}^{2}=0$ and $H.C_{1}=H.C_{2}=3$.
\end{lem}
\begin{proof}
Fix a closed embedding $S \hra \P^{4}$ given by $|H|$. 
The former statement is well-known. 
Note that $Q$ is smooth along $S$ since $S$ is a non-singular Cartier divisor of $Q$. 
Let $q$ be a quadratic form defining $Q$. 
Since $S \hra \P^{4}$ is non-degenerate, we have $\rk q \geq 3$. 
If $\rk q=3$, then $Q$ is isomorphic to an weighted projective space $\P(1,1,2,2)$. 
But it contradicts that $Q$ is smooth along $S$. 
%Indeed, $Q$ is isomorphic to an weighted projective space $\P(1^{2},2^{2})=\Proj \C[x_{0},x_{1},y_{0},y_{1}]$ and 
%the defining polynomial of $S$ can be written as 
%\[F = f(x_{0},x_{1})+\sum y_{i}g_{i}(x_{0},x_{1})+\sum y_{i}y_{j}h_{ij}(x_{0},x_{1})+e(y_{0},y_{1}),\]
%where $f,g_{i},h_{ij},e$ are homogenous polynomials of degree $6,4,2,3$. 
%But it contradicts that $Q$ is smooth along $S$ since $\Sing Q=\{x_{0}=x_{1}=0\}$. 
If $\rk q = 4$, then we have $Q=\{x_{0}x_{3}=x_{1}x_{2}\} \sb \P^{4}=\Proj \C[x_{0},\ldots,x_{4}]$. 
Set $P_{01}:=\{x_{0}=x_{1}=0\}$, $P_{02}:=\{x_{0}=x_{2}=0\}$, $H_{0}=\{x_{0}=0\}$ and $C_{i}:=P_{0i} \cap S$. 
%Then $P_{01},P_{02}$ are non-Cartier divisors satisfying $P_{01}+P_{02}=H_{0}$ that do not contain $S$. 
Note that $C_{i}$ is a divisor of $S$ since $S$ does not pass the vertex of $Q$. 
Since $C_{i}$ is a cubic curve of $P_{0i}$, we have $C_{i}^{2}=0$, $H.C_{i}=3$ and $H=C_{1}+C_{2}$. 
% and hence $p_{a}(C_{i})=1$. 
It contradicts our assumption.
\end{proof}
At the end of this subsection, we summarize a part of Mukai's theory about K3 surfaces \cite{Muk95}. 
\begin{defi}
Let $(S,H)$ be a polarized K3 surface, that is a pair of a K3 surface $S$ and an ample divisor $H$. 
We say that $(S,H)$ is \textit{Brill-Noether general} 
if $H=L+N$ for $L,N \in \Pic(S)-\{0\}$, then $h^{0}(H)>h^{0}(L)h^{0}(N)$. 
\end{defi}
\begin{prop}[{\cite[Theorem~4.7]{Muk95}}]
Let $(S,H)$ be a Brill-Noether general polarized K3 surface. Assume that $H$ is very ample, $H^{2}=2g-2$ and $g \in \{7,8,9,10\}$. Then there exists a non-singular Mukai threefold $V$ of genus $g$ such that $V$ has $S$ as an anti-canonical member. 
\end{prop}
%\begin{proof}
%By \cite[Theorem~4.7]{Muk95}, there exists a non-singular homogeneous space $\Sigma \sb \P^{N}$ such that $\P^{g} \cap \Sigma$ is isomorphic to $S$. 
%Let consider the blowing-up $\Sigma$ along $S$. 
%There exists natural morphism $\Bl_{S}\Sigma \to \P^{N-g-1}$ and the fiber $V$ is non-singular and isomorphic to a linear section of $\Sigma$ containing $S$. Namely, $V$ is a Fano threefold with index $1$ and degree $2g-2$. 
%\end{proof}
\subsection{Computations of $h^{1,2}$}\label{ss:hodge}
We can compute the Hodge numbers $h^{1,2}(X)$ of the members belonging to Case (B-i) by using the following well-known lemmas. 
\begin{lem}[{\cite[Corollary~4.12]{Kol89}}]\label{lem:3}
Let $X$ be a non-singular projective threefold and $\chi \colon X \dra X^{+}$ be a flop. Then, for all $i,j \in \Z_{\geq 0}$, $h^{i,j}(X)=h^{i,j}(X^{+})$. 
\end{lem}
\begin{lem}\label{lem:4}
Let $V$ be a smooth projective variety and $B \sb V$ a smooth closed subvariety with $\codim_{V}(B) \geq 2$. Let $\wt{V} \to V$ be the blowing-up along $B$. Then we obtain that 
\[b_{i}(\wt{V})=b_{i}(V)+\sum_{j=1}^{\codim_{V}B-1} b_{i-2j}(B), \text{ where } b_{i}(X):=\dim H^{i}(X,\Q).\]
In particular, $h^{1,2}(\wt{V})=h^{1,2}(V)+h^{0,1}(B)$ holds.
\end{lem}
%\begin{proof}
%We give a sketch of proof in well-known manners. 
%Let $E:=\Exc(\vp)$, $U_{1}$ a neighborhood of $B$ which have $B$ as a deformation retract of $U_{1}$, $U_{2}:=V-B$, 
%$\wt{U_{2}}:=\s^{-1}(U_{2})$ and 
%$\wt{U_{1}}:=\Bl_{B}U_{1}$. 
%By using the Mayer-Vietoris sequence for coverings $\{U_{1},U_{2}\}$ and $\{\wt{U_{1}},\wt{U_{2}}\}$, 
%the cokernel of the injection $H^{i}(V,\Q) \to H^{i}(\wt{V},\Q)$ is equal to the cokernel of $H^{i}(B,\Q) \to H^{i}(E,\Q)$. 
%Hence we obtain $b_{i}(\wt{V})=b_{i}(V)+\sum_{j=1}^{\codim_{V}B-1} b_{i-2j}(B)$. 
%$h^{1,2}(\wt{V})=h^{1,2}(V)+h^{0,1}(B)$ is given by the Hodge symmetry and the Serre duality and the fact that $h^{3,0}(V)$ is a birational invariant.
%\end{proof}
The Hodge number $h^{1,2}$ of a Fano threefold with Picard rank 1 is well-known \cite{Fanobook}. 
%We can obtain the following values by using Lemma \ref{lem:3}, \ref{lem:4} and \cite[Theorem~3.4.1, Theorem~4.3.3]{Fanobook}. 
\begin{lem}\label{lem:h12ofFano}
We have $h^{1,2}(B(4))=2$, $h^{1,2}(B(5))=0$, $h^{1,2}(V(9))=3$ and $h^{1,2}(V(10))=2$. 
\end{lem}
By using Lemma \ref{lem:3}, \ref{lem:4} and \ref{lem:h12ofFano}, 
we obtain the Hodge numbers $h^{1,2}(X)$ for $X$ appearing in Case (B-i) as in TABLE \ref{table:B}.
\section{Constructions}
\subsection{Case (B-i-1)}\label{ss:Bi1}
Let us construct an example belonging to Case (B-i-1). 

Let $C \sb \P^{5}$ be an elliptic curve of degree 6. 
Then there exists a linear subvariety $\P^{5} \sb \P^{7}$ such that $C=\P^{5} \cap (\P^{1})^{3}$, where $(\P^{1})^{3} \hra \P^{7}$ is the Segre embedding. 
In particular, $C$ is defined by quadratic equations. 
Thus we can take general quadric hypersurfaces $Q_{1},Q_{2}$ containing $C$ such that $V:=Q_{1} \cap Q_{2}$ is a non-singular del Pezzo threefold of degree 4. 

Let $\vp^{+} \colon X^{+}:=\Bl_{C}V \to V$ be the blowing-up along $C$ and set $H:=\vp^{+\ast}\mc O_{\P^{5}}(1)|_{V}$. 
Since $C$ is defined by quadratic equations, we obtain that $-K_{X^{+}}=2H-D$ is nef. Also we obtain that $(-K_{X^{+}})^{3}=8$ by a straightforward calculation. 
In particular, $-K_{X^{+}}$ is nef and big. 

By the classification in \cite{MM81}, $-K_{X^{+}}$ is not ample and thus 
there exists the contraction of the $K_{X^{+}}$-trivial ray $\psi \colon X^{+} \to \ol{X}$. 
$\psi^{+}$ is not a divisorial contraction by the classification in \cite{JPR05} and thus 
there exists the flop $\eta \colon X^{+} \dra X$ of $\psi$ and 
the contraction of the $K_{X}$-negative ray $\vp \colon X \to V$. 
By the classification in \cite{CM13} and \cite{JPR11}, $\vp$ is neither a divisorial contraction nor a conic bundle. Hence $\vp$ is a del Pezzo fibration. 
By the classification in \cite{JPR11}, $\vp \colon X \to \P^{1}$ is an example of Case (B-i-1). \bbox
\subsection{Case (B-i-2)}\label{ss:Bi2}
%Let us construct a Mukai threefold $V$ of genus 10, a K3 surface $S$ and a smooth curve $C$ which satisfies the following. 
%\begin{enumerate}
%\item $C \sb S$ and $S \in |-K_{V}|$. 
%\item $g_{C}=1$ and $-K_{V}.C=6$.
%\item $X^{+}:=\Bl_{C}V$ is an almost Fano. 
%\end{enumerate}
In order to construct an example belonging to Case (B-i-2), we consider the following lattice : 
\[\begin{pmatrix}
&H&C\\
H&18&6\\
C&6&0
\end{pmatrix}.\]

This is an even lattice with signature $(1,1)$. By Theorem \ref{thm:NSL}, there exists K3 surface $S$ such that $\Pic(S)=\Z \cdot H \oplus \Z \cdot C$ is isometry to this lattice. 
Moreover, we may assume that $H$ is nef and big. 

For every $D=xH+yC \in \Pic(S)=\Z \cdot H \oplus \Z \cdot C$, we have 
\[D^{2}=18x^{2}+12xy \text{ and } H.D=18x+6y.\]
Note that $D^{2}$ and $H.D$ are multiple by $6$.
In particular, $S$ has no $(-2)$-curve and hence every effective divisor on $S$ is base point free. 
Thus $H$ is very ample by Theorem \ref{thm:SDR}. 
Since $C^{2}=0$ and $H.C=6$, $C$ is effective and hence has no base point. 
Therefore, we may assume that $C$ is non-singular. 
\begin{lem}\label{lem:BN1}
This polarized K3 surface $(S,H)$ is Brill-Noether general.
\end{lem}
\begin{proof}
Let $L$ and $N$ be non-zero effective divisors such that $H=L+N$ and set $L=xH+yC$ for some $x,y \in \Z$. 
Thus we have $N=(1-x)H-yC$. 
Since $L$ and $N$ are effective, we have $H.L>0$, $H.N>0$, $C.L \geq 0$, $C.N \geq 0$, $L^{2} \geq 0$ and $N^{2} \geq 0$. By this observation, we obtain the following:
\[0 \leq x \leq 1, \quad 0 \leq 3x+y <3 \text{ and } 0 \leq 3x+2y \leq 3.\]
By a straightforward calculation, the pairs of integers satisfying this inequalities are $(x,y)=(0,1),(1,-1)$. In each case, $h^{0}(L)h^{0}(N)=10 < h^{0}(H)=11$.
\end{proof}
Therefore, there exists a Mukai threefold $V$ of genus $10$ having $S$ as an anti-canonical divisor. 
Let $\vp^{+} \colon X^{+}:=\Bl_{C}V \to V$ denotes the blowing-up of $V$ along $C$. 
Since $C \sb S$, we have $S \simeq (\vp^{+})^{-1}_{\ast}S \in |-K_{X^{+}}|$. 
It is clear that $-K_{X^{+}}$ is nef if and only if $-K_{X^{+}}|_{S}$ is nef. 
For the line bundle $-K_{X^{+}}|_{S}=H-C$, we have $(H-C)^{2}=6$ and $H.(H-C)=12$ and hence $|H-C|$ is base point free. 
Hence $-K_{X^{+}}$ is nef and $(-K_{X^{+}})^{3}=6$. 
Thus $-K_{X^{+}}$ is nef and big. 
By applying similar arguments in $\S$ \ref{ss:Bi1}, we obtain the flop $\chi \colon X^{+} \dra X$ and a del Pezzo fibration $\vp \colon X \to \P^{1}$ of degree $6$ which belongs to Case (B-i-2). \bbox
%By the classification in \cite{MM81}, $-K_{X^{+}}$ is not ample and thus there exists 
%the contraction of the $K$-trivial ray $\psi \colon X^{+} \to \ol{X}$. 
%$\psi^{+}$ is not a divisorial contraction by the classification in \cite{JPR05} and thus there exists the flop $\eta \colon X^{+} \dra X$ of $\psi$ 
%and the contraction of the $K$-negative ray $\vp \colon X \to V$. 
%By the classification in \cite{CM13}, $\vp$ is a not divisorial contraction. 
%By the classification in \cite{JPR11}, $\vp$ is a not conic bundle. 
%Hence $\vp$ is a del Pezzo fibration. 
%By the classification in \cite{JPR11}, $\vp \colon X \to \P^{1}$ is an example of Case (B-i-2). \bbox
\subsection{Case (B-i-3)}\label{ss:Bi3}
%Let construct a Mukai threefold $V$ of genus 9, a K3 surface $S$, a smooth curve $C$ which satisfies the following. 
%\begin{enumerate}
%\item $C \sb S$ and $S \in |-K_{V}|$. 
%\item $g_{C}=1$ and $-K_{V}.C=6$.
%\item $X^{+}:=\Bl_{C}V$ is a almost Fano. 
%\end{enumerate}
In order to construct an example belonging to Case (B-i-3), we consider the following lattice : 
\[\begin{pmatrix}
&H&C\\
H&16&6\\
C&6&0
\end{pmatrix}.\]
This is an even lattice with signature $(1,1)$. By Theorem \ref{thm:NSL}, there exists a K3 surface $S$ such that $\Pic(S)=\Z \cdot H \oplus \Z \cdot C$ is isometry to this lattice and $H$ is nef and big. 

For every $D=xH+yC \in \Pic(S)$ with $x,y \in \Z$, we have 
\[D^{2}=16x^{2}+12xy \text{ and } H.D=16x+6y.\]
Note that $D^{2}$ is multiple by $4$.
In particular, $S$ has no $(-2)$-curve and every effective divisor on $S$ is base point free. 
By using Theorem \ref{thm:SDR}, it is easy to see that $H$ is very ample. 
%If $D^{2}=0$, then $x(4x+3y)=0$. Hence we have $x=0$ or $4x+3y=0$ and hence we have $H.D=6y$ or $H.D=8x$. 
%Thus $H.D \neq 1,2$ for $D \in \Pic(S)$ with $D^{2}=0$. 
Since $C^{2}=0$ and $H.C=6$, $C$ is an effective divisor and hence it is base point free. Therefore, we may assume that $C$ is non-singular by Theorem \ref{thm:SDR}. 
The following lemma can be proved similarly to Lemma \ref{lem:BN1}. 
\begin{lem}\label{lem:BN2}
This polarized K3 surface $(S,H)$ is Brill-Noether general.
\end{lem}
%\begin{proof}
%Let $L,N$ be a non-zero effective divisors such that $H=L+N$. Let write $L=xH+yC$ for some $x,y \in \Z$. Then $N=(1-x)H-yC$. 
%Because $L,N$ are effective, $H.L>0$, $H.N>0$, $C.L \geq 0$, $C.N \geq 0$, $L^{2} \geq 0$ and $N^{2} \geq 0$. Thus we have 
%\[0 \leq x \leq 1, \quad 4-4x \geq 3y \geq -4x.\]
%By a straightforward calculation, 
%the pairs of integers satisfying this inequality are $(x,y)=(0,1),(1,-1)$. In each case, $h^{0}(L)h^{0}(N)=8 < h^{0}(H)=9$.
%\end{proof}
Therefore, there exists a Mukai threefold $V$ of genus $9$ having $S$ as an anti-canonical divisor. 
Let $\vp^{+} \colon \Bl_{C}V=:X^{+} \to V$ denotes the blowing-up of $V$ along $C$. By a straightforward calculation, $(-K_{X^{+}})^{3}=4$ holds. 
By the similar argument in $\S$ \ref{ss:Bi2}, $-K_{X}$ is nef and big and we obtain the flop $\chi \colon X^{+} \dra X$ and a del Pezzo fibration $\vp \colon X \to \P^{1}$ of degree $6$ which belongs to Case (B-i-3). \bbox
%Let us prove that $-K_{X^{+}}$ is nef and big. 
%It is enough to show that $-K_{X^{+}}|_{S_{X^{+}}}$ is nef where $S_{X^{+}}:=\vp^{-1}_{\ast}S \simeq S$. 
%It is equivalent to show that $H-C$ is nef on $S$. 
%Now, we have $(H-C)^{2}=4$ and $H.(H-C)=10$. 
%Thus $|H-C|$ is base point free and therefore $-K_{X^{+}}$ is nef. 
%Since $(-K_{X^{+}})^{3}=4$, $-K_{X^{+}}$ is nef and big. 
%By applying similar arguments in $\S$ \ref{ss:Bi1}, 
\subsection{Quadric fibrations with tri-sections}\label{ss:dP6dP8}
In this subsection, we prove Proposition \ref{prop:1}. 
This proposition plays an important role on the constructions of the examples belonging to Case (A-1) , (B-ii) and (B-iii). 
In order to prove Proposition \ref{prop:1}, we prepare the following two lemmas. 
%At first, let us show Lemma \ref{lem:5}, \ref{lem:6} which need to prove Proposition \ref{prop:1}.
\begin{lem}\label{lem:5}
Let $Y$ be a smooth projective threefold and $S$ a smooth projective surface and $g \colon Y \to S$ and $\vp \colon S \to \P^{1}$ morphisms. 
Assume $g$ is a contraction of $K_{Y}$-negative ray over $\P^{1}$. Set $\pi:=\vp \circ g$. 
\[\xymatrix{
Y\ar[r]^{g} \ar[rd]_{\pi} & S \ar[d]^{\vp} \\
&\P^{1}
}\]
If $\rho(Y)=3$ and $\pi^{\ast}\mc O(1).(-K_{Y})^{2}>0$, then $S \simeq \F_{n}$ and $\vp \colon S=\F_{n} \to \P^{1}$ is the $\P^{1}$-bundle structure. 
\end{lem}
\begin{proof}
Let $\D \sb S$ be a discriminant divisor of the conic bundle $g \colon Y \to S$. 
Then $g_{\ast}(-K_{Y})^{2} \equiv -(4K_{S}+\D)$ holds \cite[Corollary~4.6]{MM86}. 
Let $F$ be a $\vp$-fiber. 
By assumptions of this lemma, we have that 
$-F.(4K_{S}+\D)=F.g_{\ast}(-K_{Y})^{2}=g^{\ast}F.(-K_{Y})^{2}=\pi^{\ast}\mc O(1).(-K_{Y})^{2}>0$. 
Hence we have $4K_{S}.F<-\D.F \leq 0$, which means $-K_{S}$ is $\vp$-ample. 
Since $\rho(S)=2$, all fibers of $\vp$ are isomorphic to $\P^{1}$.
\end{proof}
\begin{lem}\label{lem:6}
Let $f \colon X \to S$, $f' \colon X' \to S'$ be Mori fiber spaces and 
$\Phi \colon X \dra X'$ an isomorphism in codimension 1. 
Assume there exists a rational map $\vp \colon S \dra S'$ such that the following diagram is commutative: 
\[\xymatrix{
X\ar[d]_{f} \ar@{-->}[r]^{\Phi} & X' \ar[d]^{f'} \\
S \ar@{-->}[r]_{\vp} & S'.
}\]
Then $\Phi$ and $\vp$ are isomorphic. 
\end{lem}
\begin{proof}
See \cite[Proposition~3.5]{Cor95}. 
\end{proof}
\begin{prop}\label{prop:1}
Let $W$ be a smooth projective threefold with $\rho(W)=2$, 
$\pi \colon W \to \P^{1}$ a quadric fibration, 
$B \sb W$ a smooth projective curve and 
$\t \colon Z:=\Bl_{B}W \to W$ the blowing-up along $B$. 
We assume the following condition $(\dag_{6})$ for the pair $(\pi \colon W \to \P^{1},B)$.
\[(\dag_{6}) \cdots
\left\{ \begin{array}{ll}
\deg (\pi|_{B} \colon B \to \P^{1})=3 \text{ and }\\
-K_{Z} \text{ is $p$-nef and $p$-big with } p := \pi \circ \t \colon Z \to \P^{1}. 
\end{array} \right. \]
Then the following hold.
\begin{enumerate}
\item The morphism from $Z$ to the relative anti-canonical model over $\P^{1}$ 
\[ \Psi \colon Z \to \ol{Z}:=\Proj_{\P^{1}}\bigoplus_{n \geq 0} p_{\ast}\mc O_{Z}(-nK_{Z})\]
is a small contraction or an isomorphism. 
\item Define a variety $Y$ and a birational map $\Phi \colon Z \dra Y$ over $\P^{1}$ as follows: 
\begin{itemize}
\item If $-K_{Z}$ is $p$-ample, we set $Y:=Z$ and $\Phi:=\id_{Z} \colon Z \to Y$, or 
\item if $-K_{Z}$ is not $p$-ample, let $\Phi \colon Z \dra Y$ be the flop of $\Psi \colon Z \to \ol{Z}$. 
\end{itemize}

Let $q \colon Y \to \P^{1}$ be a structure morphism onto $\P^{1}$ and $\mu \colon Y \to X$ a $K_{Y}$-negative ray contraction over $\P^{1}$. If $-K_{Z}$ is $p$-ample, we choose the other ray which is different to the one corresponding to $\t \colon Z \to W$. Let $\vp \colon X \to \P^{1}$ denotes the structure morphism onto $\P^{1}$. 

Then $\mu$ is the blowing-up along a $\vp$-section $C$ and $\vp \colon X \to \P^{1}$ is a del Pezzo fibration of degree 6. 
\[\xymatrix{
&D \ar[ld]& \ar@{}[l]|{\sb} \ar[ld]_{\mu}Y=\Bl_{C}X \ar[dd]^{q}\ar@{<--}[r]^{\Phi}& Z=\Bl_{B}W \ar@{}[r]|{\sp} \ar[rd]^{\t} \ar[dd]_{p}& E \ar[rd]& \\
C&\ar@{}[l]|{\sb} X \ar[rd]_{\vp}& && W \ar@{}[r]|{\sp} \ar[ld]^{\pi} &B&(\bigstar_{6}) \\
&&\P^{1} \ar@{=}[r]&\P^{1}&&
} \]
\item Let $F_{Y}$ be a general $q$-fiber and we set $D:=\Exc(\mu)$, $E:=\Exc(\t)$ and $E_{Y}:=\Phi_{\ast}E$. Then we obtain 
\[D \equiv \frac{1}{2}(-K_{Y})-\frac{1}{2}E_{Y}+zF_{Y}.\]
Moreover, the following equalities hold: 
\begin{align*}
(-K_{X})^{3}&=\frac{3(-K_{W})^{3}-16g_{B}-32}{4}, \\
-K_{X}.C&=\frac{8 (-K_{W}).B-24g_{B}-(-K_{W})^{3}-32}{8} \text{ and } \\
z&=\frac{4(-K_{W}).B-8g_{B}-(-K_{W})^{3}}{8}.
\end{align*}
\end{enumerate}
\end{prop}
\begin{proof}
\begin{enumerate}
\item Assume $\Psi \colon Z \to \ol{Z}$ is a divisorial contraction. Set $D:=\Exc(\Psi)$. 
Let $F_{Z}$ be a general $p$-fiber and $F_{W}$ be a general $\pi$-fiber such that $\t(F_{Z})=F_{W}$. 
Now $\t|_{F_{Z}} \colon F_{Z} \to F_{W}\simeq \P^{1} \times \P^{1}$ is the blowing-up at reduced three points. 

Let us prove that $D|_{F_{Z}}$ is a disjoint union of $(-2)$-curves $l_{1},\ldots,l_{n}$ and $n \in \{1,2\}$. 
Let $e_{1},e_{2},e_{3}$ be exceptional curves of $\t|_{F_{Z}}$ and 
set $f_{1}:=(\t|_{F_{Z}})^{\ast}\mc O_{\P^{1} \times \P^{1}}(1,0)$ and $f_{2}:=(\t|_{F_{Z}})^{\ast}\mc O_{\P^{1} \times \P^{1}}(0,1)$. 
For every irreducible curve $l \sb F_{Z}$, 
$l=a_{1}f_{1}+a_{2}f_{2}-\sum_{j=1}^{3} b_{j}e_{j}$
holds for some $a_{i},b_{j} \in \Z_{\geq 0}$. 
If $l$ is contracted by $\t|_{F_{Z}}$, then $l$ is $(-2)$-curve and we have the following: 
\[0=2(a_{1}+a_{2})-\sum_{j=1}^{3}b_{j} \text{ and } -2=2a_{1}a_{2}-\sum_{j=1}^{3}b_{j}^{2}.\]
By the inequality 
\[2+2a_{1}a_{2}=\sum_{j=1}^{3}b_{j}^{2} \geq 3 \left( \frac{\sum_{j=1}^{3}b_{j}}{3} \right)^{2}=\frac{4}{3}(a_{1}+a_{2})^{2},\]
we have $(a_{1},a_{2})=(0,1)$ or $(1,0)$. 
Thus we have $(b_{1},b_{2},b_{3})=(1,0,0)$ or $(0,1,0)$ or $(0,0,1)$ 
and hence $l$ is the proper transform of a line passing through exactly two points of the three points. 
By our assumption, $-K_{F_{Z}}$ is nef and hence the three points are not collinear. 
Thus the number of lines passing through exactly two points of the three points is at most two. 
Moreover, if there exists two lines passing through two points of the three points, then the proper transforms of those do not meet since the two lines meet transversally at one point of the three point.

Denote $D \equiv x(-K_{Z})+yE+zF_{Z}$ with $x,y,z \in \Q$. 
Then $D|_{F_{Z}} \equiv x(-K_{F_{Z}})+y(E|_{F_{Z}})$ and hence we have the following: 
\begin{align*}
0&=(-K_{F_{Z}})(D|_{F_{Z}})=5x+3y \text{ and } \\
-2n&=(D|_{F_{Z}})^{2}=5x^{2}-3y^{2}+6xy. 
\end{align*}
Thus we have $x=\pm \frac{\sqrt{15n}}{10}$. It contradicts $x \in \Q$ and $n \in \{1,2\}$. Therefore, $\Psi$ is an isomorphism or a small contraction. 
\item 
%Set $Y$ and $\Phi \colon Z \dra Y$ as in the statement. 
%Let $q \colon Y \to \P^{1}$ be a structure morphism. 
For a general $p$-fiber $F_{Z}$ and the proper transform $F_{Y}:=\Phi_{\ast}F_{Z}$, 
the birational map $\Phi|_{F_{Z}} \colon F_{Z} \dra F_{Y}$ is an isomorphism. 
We abbreviate $F_{Y},F_{Z}$ to $F$ for simplicity. 

Let us prove that $\mu \colon Y \to X$ is the blowing-up along a non-singular curve $C$. 
Since $\rho(Y)=3$, then $\dim X \geq 2$. 
If $\dim X=2$, $X$ is isomorphic to a Hirzebruch surface $\F_{n}$ by Lemma \ref{lem:5}.
Set $D=g^{\ast}\mc O_{\F_{n}}(1)$ and $D \equiv x(-K_{Z})+yE+zF_{Z}$ with $x,y,z \in \Q$. Then we have the following equations: 
\begin{align*}
2&=(-K_{Y}).D.F=(-K_{F})D|_{F}=5x+3y \text{ and } \\
0&=D^{2}F=(D|_{F})^{2}=5x^{2}-3y^{2}+6xy.
\end{align*}
Hence $x=\frac{4 \pm \sqrt{6}}{10}$, $ y=\mp \frac{1}{\sqrt{6}}$ and it contradicts $x,y \in \Q$. 
Hence $\mu$ is a divisorial contraction. 
Set $D=\Exc(\mu)$. If $\mu(D)$ is a point, $D$ does not meet the general fibers $F$ of $Y \to \P^{1}$. 
Thus we have the following: 
\begin{align*}
0&=(-K_{Y}).D.F=(-K_{F})D|_{F}=5x+3y \text{ and } \\
0&=D^{2}F=(D|_{F})^{2}=5x^{2}-3y^{2}+6xy.
\end{align*}
Hence we have $x=y=0$ and $D=zF$. In particular, $D$ or $-D$ is nef which is impossible. 
Therefore, $\dim \mu(D)=1$ and $\mu$ is the blowing-up along the smooth curve $C:=\mu(D)$. 

Let $F_{X}$ denotes a general $\vp$-fiber. 
Set $D \equiv x(-K_{Z})+yE+zF$ with $x,y,z \in \Q$ and $m:=F_{X}.C \in \Z_{\geq 0}$. 
Since $(-K_{F_{Y}})^{2}=5$ and $(-K_{F_{X}})^{2} \leq 9$, we have $0 \leq m \leq 4$. 
By a calculation of some intersection numbers, we have the following: 
\begin{align*}
m&=(-K_{Y})DF=5x+3y, \\
-m&=D^{2}F=5x^{2}-3y^{2}+6xy \text{ and } \\
\Z &\ni 3x-3y=EFD. 
\end{align*}
By these equality, we have the following: 
\begin{align*}
&x=\frac{4m \pm \sqrt{6m^{2}+30m}}{20}, y=\mp \sqrt{\frac{m^{2}+5m}{24}} \text{ and } \\
&3x-3y=\frac{3m \pm \sqrt{24m(m+5)}}{32} \in \Z.
\end{align*}
Then the possibilities of $(m,x,y)$ are as follows: 
\[(m,x,y)=(0,0,0),\left(1,\frac{1}{2},-\frac{1}{2} \right) , (3,0,1).\]
It is impossible that $x=y=0$ as we have seen. 
If $(x,y)=(0,1)$, then we have $D|_{F}=E|_{F}$ for a general $F$ and hence the birational map $\psi \colon W \dra X$ is isomorphic in codimension 1. 
Since $F_{W}$ is birationally transformed into $F_{X}$ by the birational map $\psi \colon W \dra X$, $\psi$ is an isomorphism over $\P^{1}$ by Lemma \ref{lem:6}. 
Moreover, we have $\psi(C)=B$. 
In the case that $-K_{Z}$ is $p$-ample, it contradicts that the rays which corresponding $\t$, $\mu$ are different. 
In the case that $-K_{Z}$ is not $p$-ample, it contradicts that $\Psi \colon Y \dra Z$ is not an isomorphism over $\P^{1}$. 
Therefore, we obtain that $(m,x,y)=\left(1,\frac{1}{2},-\frac{1}{2} \right)$ and hence $(-K_{F_{X}})^{2}=6$. 
Therefore, $\vp|_{C} \colon C \to \P^{1}$ is isomorphic and $\vp \colon X \to \P^{1}$ is a del Pezzo fibration of degree $6$. 
\item Since $x=\frac{1}{2}$, $y=-\frac{1}{2}$, $(-K_{Z})^{3}=(-K_{Y})^{3}$, $(-K_{Z})^{2}D=(-K_{Y})^{2}D_{Y}$ and $(-K_{Z})D^{2}=(-K_{Y})D_{Y}^{2}$, we obtain the following: 
\begin{align*}
&(-K_{X})^{3}-2(-K_{X}).C-2=(-K_{W})^{3}-2(-K_{W}).B+(2g_{B}-2), \\
&(-K_{X}).C+2=\frac{(-K_{W})^{3}-2(-K_{W}).B+(2g_{B}-2) }{2}-\frac{(-K_{W}).B+2-2g_{B}}{2}+5z \text{ and } \\
&-2=\frac{(-K_{W})^{3}-2(-K_{W}).B+(2g_{B}-2)}{4} +\frac{2g_{B}-2}{4}-\frac{(-K_{W}).B+2-2g_{B}}{2}+2z.
\end{align*}
The last equalities in Proposition \ref{prop:1} are obtained by solving these equations. \qedhere
\end{enumerate}
\end{proof}
\subsection{Case (B-ii)}\label{ss:Bii}
In order to construct an example belonging to Case (B-ii), we consider the following lattice: 
\[\begin{pmatrix}
& H & F & B \\
H & 6 & 4 & 6 \\
F & 4 & 0 & 3 \\
B & 6 & 3 & 2
\end{pmatrix}.\]

This is an even lattice with signature $(1,2)$. By virtue of Theorem \ref{thm:NSL}, there exists a K3 surface $S$ such that $\Pic(S)=\Z \cdot [H] \oplus \Z \cdot [F] \oplus \Z \cdot [B]$ is isometry to this lattice. Moreover, we can assume that $H$ is nef and big. 

Set $C_{2}=H-F$ and $C_{4}=H+F-B$. 
Then we have $(C_{i})^{2}=-2$ and $H.C_{i}=i$ for $i \in \{2,4\}$. 
Since $H$ is nef, $C_{2},C_{4}$ are $(-2)$-curves. 

Let $C$ be a divisor of $S$. Then there exists $x,y,z \in \Z$ such that $C=xH+yF+zB$ and we have
\begin{align*}
C^{2}=6x^{2}+2z^{2}+8xy+12xz+6yz \text{ and } H.C=6x+4y+6z. 
\end{align*}
By solving these equations for $x,y$, we have the following: 
\begin{align*}
&x=\frac{2(H.C)-9z \pm \sqrt{4(H.C)^{2}-24C^{2}-87z^{2}}}{12} \text{ and } \\
&y=\frac{-3z \mp \sqrt{4(H.C)^{2}-24C^{2}-87z^{2}}}{8}.
\end{align*}
Hence $x,y,z \in \Z$ and $H.C$ is even, 
we can prove that the following (i)-(iv) by a straightforward calculation. 
\begin{itemize}
\item[(i)] It is impossible that $C^{2}=0$ and $H.C=1$. 
\item[(ii)] It is impossible that $C^{2}=0$ and $H.C=2$. 
\item[(iii)] It is impossible that $C^{2}=0$ and $H.C=3$. 
\item[(iv)] If $C^{2}=-2$ and $H.C \leq 9$, then $C=C_{2}$ or $C_{4}$. 
\end{itemize}

\begin{lem}\label{lem:Bii}
\begin{enumerate}
\item $H$ is very ample and $S$ is embedded to $\Q^{3}$ by $|H|$ as an anti-canonical member. 
\item The linear systems $|F|$ and $|B|$ have non-singular members and the general members meet transversely in three point. 
\item $2H+F-B$ is nef. 
\end{enumerate}
\end{lem}
\begin{proof}
\begin{enumerate}
\item By (i),(ii),(iv) and Theorem \ref{thm:SDR}, $H$ is very ample. 
By (iii) and Lemma \ref{lem:2}, $S$ is embedded in $\Q^{3}$ as an anti-canonical divisor. 
\item By Theorem \ref{thm:SDR}, it is enough to show that $|F|$ and $|B|$ are movable.
Let us prove that $|F|$ is movable. 
Since $H.F=4$ and $F^{2}=0$, $F$ is effective. 
Let $M$ be the movable part of $|F|$ and $N$ the fixed part. 
Then we have $4=H.F>H.N$. 
By Lemma \ref{lem:1} and (iv), we have $N=aC_{2}$ for some $a \in \Z_{\geq 0}$. Since $M$ is movable, we have $0 \leq M.F=(F-aC_{2}).F=-4a$ and hence $a=0$. 
Therefore, $F=M$ is movable. 

Let us prove that $|B|$ is movable. 
By the similar argument, there exists $a,b \in \Z_{\geq 0}$ such that $B=M+aC_{2}+bC_{4}$, where $M$ is the movable part of $|B|$. 
Since $M$ is movable, we have $0<M^{2}=2-2a^{2}-2b^{2}-6a-14b+6ab$ and it is easy to find that $(a,b)=(0,0)$. 
Hence $B=M$ is movable. 
%Since $B^{2}=2$ and $H.B=6$, $B$ is effective. 
%By the similar argument, we can write $B=M+aC_{2}$ for some $a \in \Z_{\geq 0}$, where $M$ is the movable part of $B$. 
%Since $M$ is movable, $0 \leq M.F=(B-aC_{2}).F=3-4a$ and hence $a=0$. 
\item Since $(2H+F-B)^{2}=12$ and $H.(2H+F-B)=10$, we obtain that $2H+F-B$ is effective. 
Let $M$ be the movable part and $N$ the fixed part of $2H+F-B$. 
Since $10=H.(2H+F-B)>H.N$ and Lemma \ref{lem:1} and (iv), there exists $a,b \in \Z_{\geq 0}$ such that $2H+F-B=M+aC_{2}+bC_{4}$. 
Since $(2H+F-B).C_{2}=5$ and $(2H+F-B).C_{4}=2$, $2H+F-B$ is nef. 
\qedhere
\end{enumerate}
\end{proof}
By Lemma \ref{lem:Bii}~(1), $S$ is a smooth anti-canonical member of $\Q^{3}$ and 
$C_{2} \sb S$ is a conic in $\Q^{3}$. 
Let $\s \colon W:=\Bl_{C_{2}}\Q^{3} \to \Q^{3}$ denotes the blowing-up of $\Q^{3}$ along $C_{2}$. 
It is well-known that $W$ is a Fano threefold and there exists a quadric fibration $\pi \colon W \to \P^{1}$ that is given by the complete linear system $|\s^{\ast}\mc O_{\Q^{3}}(1)-\Exc(\s)|$ \cite[No.29~of~Table 3]{MM81}. 
Note that $S_{W} \simeq S$ and $B_{W} \simeq B$. 
In particular, we can treat divisors of $S_{W}$ as if those were divisors of $S$. 
Then we obtain $\pi^{\ast}\mc O_{\P^{1}}(1)|_{S_{W}}=\mc O_{S}(F)$ for $F=H-C_{2}$ and thus $\deg(\pi|_{B_{W}} \colon B_{W} \to \P^{1})=3$. 
Let $Z:=\Bl_{B_{W}}W$ denotes the blowing-up of $W$ along $B_{W}$. 
Note that $S_{Z} \simeq S$ is a member of $|-K_{Z}|$ and $(-K_{Z})|_{S_{Z}}=\mc O_{S}(2H+F-B)$. 
The nefness of $-K_{Z}$ follows by Lemma \ref{lem:Bii}~(4). 
Moreover, we have $(-K_{Z})^{3}=12$ and hence $-K_{Z}$ is nef and big. 
By these arguments, $(\pi \colon W \to \P^{1},B)$ satisfies the condition $(\dag_{6})$ in Proposition \ref{prop:1}. 
Therefore, we obtain a del Pezzo fibration $\vp \colon X \to \P^{1}$ of degree 6 satisfying the following : 
\[(-K_{X})^{3}=14,\quad -K_{X}.C=0 \text{ and } D \equiv \frac{1}{2}(-K_{X})-\frac{1}{2}E-\frac{1}{2}F.\]
Since $Y$ is almost Fano and $-K_{X}.C=0$, $X$ is almost Fano. By TABLE \ref{table:A} and TABLE \ref{table:B}, $X$ belongs to Case (B-ii). 
By Lemma \ref{lem:3} and \ref{lem:4}, we have $h^{1,2}(X)=2$. \bbox
\subsection{Case (B-iii-1)}\label{ss:Biii1}
Our construction of an example of Case (B-iii-1) does not need Theorem \ref{thm:NSL}. 
Let $S_{0} \sb \P^{3}$ be a smooth cubic surface and $\e \colon S_{0} \to \P^{2}$ the blowing-up at 6 points in general position. 
Let $h$ (resp. $e_{1},\ldots,e_{6}$) denotes $\s^{\ast}\mc O_{\P^{2}}(1)$ (resp. $\e$-exceptional curves). 
Consider the following curves of $S_{0}$: 
\[B \in \left| 3h-\sum_{i=1}^{5} e_{i} \right| \text{ and } \G \in \left| 3h-\sum_{i=2}^{6} e_{i} \right|.\]
Note that general $B$ and $\G$ are degree 4 elliptic curves in $\P^{3}$ and meet transversely in 5 points. 

Let $\s \colon W:=\Bl_{\G}\P^{3} \to \P^{3}$ denotes the blowing-up of $\P^{3}$ along $\G$ and $G$ denotes the $\s$-exceptional divisor. 
It is well-known that 
$W$ is a Fano threefold and 
a morphism $\pi \colon W \to \P^{1}$ that is given by the complete linear system $|\s^{\ast}\mc O_{\P^{3}}(2)-\Exc(\s)|$ is 
a quadric fibration \cite[No.25~of~Table 3]{MM81}. 
Let $\t \colon Z:=\Bl_{B_{W}}W \to W$ denotes the blowing-up of $W$ along $B_{W}$. Set $E:=\Exc(\t)$ and $H:=f^{\ast}\s^{\ast}\mc O_{\P^{3}}(1)$. 
\begin{lem}
$-K_{Z}$ is nef and $(-K_{Z})^{3}=10$. 
% In particular, $-K_{Z}$ is nef and big. 
\end{lem}
\begin{proof}
$(-K_{Z})^{3}=10$ is followed by a straightforward calculation. 
Note that $-K_{Z}=4H-G_{Z}-E$ and $S_{0,Z}=3H-G_{Z}-E$. 
Since $-K_{Z}=H+S_{0,Z}$ and $H$ is nef, it is enough to show that $-K_{Z}|_{S_{0,Z}}$ is nef. 
Under an identification of $S_{0,Z} \simeq S_{0}$, we have 
\begin{align*}\label{eq:Biii1}
(-K_{Z})|_{S_{0,Z}}=4(-K_{S_{0}})-\G-B=2(-K_{S_{0}})-(e_{1}+e_{6}).
%=6h-3(e_{1}+e_{6})-2(e_{2}+\cdots+e_{5})
\end{align*}
This divisor is nef and we are done. 
\end{proof}
By these arguments, $(\pi \colon W \to \P^{1},B)$ satisfies the condition $(\dag_{6})$ in Proposition \ref{prop:1}. 
Hence we obtain the diagram $(\bigstar_{6})$. 
Because of $(-K_{W})^{3}=32$, $(-K_{W}).B=11$ and $g_{B}=1$, we obtain a del Pezzo fibration $\vp \colon X \to \P^{1}$ of degree 6 satisfying the following: 
\[D \equiv \frac{1}{2}(-K_{Y})-\frac{1}{2}E_{Y}+\frac{1}{2}F_{Y},\quad (-K_{X})^{3}=12 \text{ and } -K_{X}.C=0\]
Since $Y$ is almost Fano and $-K_{X}.C=0$, $X$ is almost Fano. By TABLE \ref{table:A} and TABLE \ref{table:B}, $X$ belongs to Case (A-1) or Case (B-iii-1). 

Let us prove that $X$ belongs to Case (B-iii-1). 
Now we obtain that $F_{W}=2H-G$ and $-K_{Z}=4H-G_{Z}-E$, where $F_{W}$ is a general fiber of $\pi \colon W \to \P^{1}$. Hence we obtain 
\[D_{Z} \equiv \frac{1}{2}(4H-G_{Z}-E)-\frac{1}{2}E+\frac{1}{2}(2H-G_{Z})=3H-G_{Z}-E =S_{0,Z}.\]
It means that $D=\Exc(\mu)=S_{0,Y}$. 
\begin{lem}
The $K_{X}$-trivial ray contraction of $X$ is a small contraction. 
\end{lem}
\begin{proof}
Let $\g \sb X$ be a integral curve such that $K_{X}.\g=0$ and $\g \neq C$. 
Since $-K_{Y}$ is nef, we have 
$0 \geq K_{Y}.\g_{Y}=(g^{\ast}K_{X}+D).\g_{Y}=D.\g_{Y}$. 
Since $\g \neq C$, $\g_{Y}$ is not contained in $D$. 
Thus $-K_{Y}.\g_{Y}=D.\g_{Y}=0$ holds. 
Since $\g_{Y}$ is not contracted by $Y \to \P^{1}$, $\g_{Y}$ is not a flopped curve of $\Phi \colon Z \dra Y$. 
In particular, we have $D_{Z}.\g_{Z} \geq 0$. 
On the other hand, we have $0=-K_{Z}.\g_{Z}=(H+D_{Z}).\g_{Z}$. 
Hence we have $H.\g_{Z}=D_{Z}.\g_{Z}=0$. 

Therefore, $\g_{Z}$ is a fiber of $G \to \G$ or a fiber of $E \to B$. 
Since $-K_{Z}|_{E}$ is relatively ample for $E \to B$, 
$\g_{Z}$ is a fiber of $G \to \G$. 
The fiber $l$ of $G \to \G$ satisfying $-K_{Z}.l=0$ is the fiber of one of the 5 points $\G \cap E$. Therefore, the $K_{X}$-trivial curves are exactly the proper transformations of these fibers and $C$. In particular, the number of the $K_{X}$-trivial curves is finite. 
\end{proof}
Hence $X$ belongs to Case (B-iii-1). 
By using Lemma \ref{lem:3} and \ref{lem:4}, we have 
$h^{1,2}(X)=3$. 
%h^{1,2}(Y)=h^{1,2}(Z)=h^{1,2}(W)+1=h^{1,2}(\P^{3})+2=2 
\bbox
\subsection{Case (A-1)}\label{ss:A1}
The construction of an example belonging to Case (A-1) is obtained by the slightly modified manners in Case (B-iii-1). 
Let $\G \sb S_{0}$, $\s \colon W \to \P^{3}$, $G=\Exc(\s)$ be as in $\S$ \ref{ss:Biii1}. 
Recall that $W$ has a quadric fibration structure $\pi \colon W \to \P^{1}$ and a general $\pi$-fiber $F_{W}$ is linearly equivalent to $\s^{\ast}\mc O_{\P^{3}}(2)-G$. 
The $\s$-exceptional divisor $G=\P(\mc E)$ has a $\P^{1}$-bundle structure $\s|_{G} \colon G \to \G$, where $\mc E=N_{\G}W^{\vee}$. Let $h=\mc O_{\P(\mc E)}(1)$ be a tautological bundle. 

Since $\G$ is an elliptic curve of degree 4, $\G$ is a complete intersection of two quadrics. 
Hence $\mc E \simeq \mc L_{1} \oplus \mc L_{2}$ for some line bundles $\mc L_{1},\mc L_{2}$ of degree $-8$. 
For an arbitrary line bundle $\mc M$ of degree $11$ on $\G$, $\mc L_{i} \otimes \mc M$ is very ample for $i \in \{1,2\}$ and hence $\mc E \otimes \mc M$ is globally generated. 
Therefore, the linear system $|h+(\s|_{G})^{\ast}\mc M|$ on $G=\P(\mc E)$ has a non-singular member $B$. 
Since $B$ is $(\s|_{G})$-section, $g_{B}=1$. Moreover, we have 
\[\s^{\ast}\mc O_{\P^{3}}(1).B=4, \quad G.B=5,\quad F_{W}.B=3 \text{ and } (-K_{W}).B=11.\]
Let $\t \colon Z:=\Bl_{B}W \to W$ be the blowing-up of $W$ along $B$, $E$ the $\t$-exceptional divisor. Set $H:=\t^{\ast}\s^{\ast}\mc O(1)$. 
Note that we have 
\[G_{Z}=\t^{\ast}G-E \text{ and } -K_{Z}=4H-\t^{\ast}G-E=4H-G_{Z}-2E.\]
\begin{lem}
$-K_{Z}$ is nef and $(-K_{Z})^{3}=10$. 
%In particular, $-K_{Z}$ is nef and big. 
\end{lem}
\begin{proof}
By a straightforward calculation, it is easy to see that $(-K_{Z})^{3}=10$. 
Let us prove the nefness of $-K_{Z}$. 
Since $-K_{Z}=4H-G_{Z}-2E=2(2H-G_{Z}-E)+G_{Z}=2F_{Z}+G_{Z}$, %we have $\Bs|2H-E| \sb G_{Z}$ and hence $\Bs|-K_{Z}| \sb G_{Z}$. 
it is enough to show the nefness of $-K_{Z}|_{G_{Z}}$. 
Since $G_{Z} \simeq G \simeq \P(\mc E)$, we can describe as follows: 
\[H|_{G}=(\s|_{G})^{\ast}\mc M_{4}, \quad G|_{G}=-h \text{ and } E|_{G_{Z}}=h+(\s|_{G})^{\ast}\mc M_{11}.\]
Here, $\mc M_{i}$ denotes some line bundle on $\G$ of degree $i$. 
In this notation, we have 
$(-K_{Z})|_{G_{Z}}=
%(4H-G-E)|_{G_{Z}}=
(\s|_{G})^{\ast}\mc M_{5}$, which complete the proof. 
\end{proof}
Therefore, $(\pi \colon W \to \P^{1},B)$ satisfies the condition $(\dag_{6})$ in Proposition \ref{prop:1} and hence we obtain the diagram $(\bigstar_{6})$ and a del Pezzo fibration $\vp \colon X \to \P^{1}$ of degree 6 satisfying the following: 
\[D \equiv \frac{1}{2}(-K_{Y})-\frac{1}{2}E_{Y}+\frac{1}{2}F_{Y},\quad (-K_{X})^{3}=12 \text{ and } -K_{X}.C=0\]
Since $Y$ is almost Fano and $-K_{X}.C=0$, $X$ is almost Fano. 
Therefore, $X$ belongs to Case (A-1) or Case (B-iii-1). 
\begin{lem}
The $K_{X}$-trivial ray contraction of $X$ is crepant. 
\end{lem}
\begin{proof}
For the general fiber $f$ of $G=\P(\mc E) \to \G$, we have $-K_{Z}.f_{Z}=0$. 
Since $F_{Z}.f_{Z}=\t^{\ast}(2H-G).f_{Z}=1$, $f_{Z}$ is not contracted by $p \colon Z \to \P^{1}$. 
Since general $f_{Z}$ does not meet the flopping curves of $\Phi$, we have $-K_{Y}.f_{Y}=0$. 
Since $D$ is linearly equivalent to $3H-G-E$, $D.f_{Y}=0$ and hence $f_{Y}$ is not contracted by $\mu \colon Y \to X$. 
Therefore, $\mu(f_{Y})=f_{X}$ is a curve such that $-K_{X}.f_{X}=0$ and this argument tells us that $X$ has 
%at least 1 dimensional family 
infinitely many $K_{X}$-trivial curves. 
\end{proof}
Hence $X$ belongs to Case (A-1). By similarly arguments in $\S$ \ref{ss:Biii1}, we obtain $h^{1,2}(X)=2$. \bbox

\subsection{Case (B-iii-2)}\label{ss:Biii2}
In order to construct an example belonging to Case (B-iii-2), we consider the following lattice: 
\[\begin{pmatrix}
&H&\G&B \\
H &6&2&9 \\
\G&2&-2&6 \\
B&9&6&6
\end{pmatrix}.\]

This is an even lattice with signature $(1,2)$. By virtue of Theorem \ref{thm:NSL}, there exists a K3 surface $S$ such that $\Pic(S)=\Z \cdot [H] \oplus \Z \cdot [F] \oplus \Z \cdot [B]$ is isometry to this lattice. Moreover, we can assume that $H$ is nef and big. 

Set $C_{2}:=\G$ and $C_{5}:=3H-2\G-B$. 
Note that $C_{i}$ is a $(-2)$-curve since $(C_{i})^{2}=-2$ and $H.C_{i}=i>0$. 

Let $C=xH+y\G+zB$ be a divisor on $S$ with $x,y,z \in \Z$. Then we have 
\[C^{2}=6x^{2}-2y^{2}+6z^{2}+4xy+18xz+12yz \text{ and }H.C=6x+2y+9z.\]
By solving these equations for $x$ and $y$, we obtain the following: 
\begin{align*}
&x=\frac{4H.C-45z \pm \sqrt{4(H.C)^{2}-24C^{2}-99z^{2}}}{24} \text{ and }\\
&y=\frac{9z-\sqrt{4(H.C)^{2}-24C^{2}-99z^{2}}}{8}.
\end{align*}
Hence $x,y,z \in \Z$, 
we can prove that the following (i)-(iv) by a straightforward calculation. 
\begin{itemize}
\item[(i)] It is impossible that $C^{2}=0$ and $H.C=1$. 
\item[(ii)] It is impossible that $C^{2}=0$ and $H.C=2$. 
\item[(iii)] It is impossible that $C^{2}=0$ and $H.C=3$. 
\item[(iv)] If $C^{2}=-2$ and $H.C \leq 10$, then $C=C_{2}$ or $C_{5}$. 
\end{itemize}
\begin{lem}\label{lem:Biii2}
\begin{enumerate}
\item $H$ is very ample and $S$ is embedded to $\Q^{3}$ by $|H|$ as an anti-canonical member. 
\item The general member of $|B|$ is non-singular.% curve of genus 4. 
\item $3H-\G-B$ is nef. 
\end{enumerate}
\end{lem}
\begin{proof}
\begin{enumerate}
\item It can be proved similarly to Lemma \ref{lem:Bii}~(1). 
%By (i),(ii),(iv) and Theorem \ref{thm:SDR}, $H$ is very ample. 
%By (iii) and Lemma \ref{lem:2}, $S$ is embedded in $\Q^{3}$ as an anti-canonical member. 
\item %Let us prove that $|B|$ has non-singular member. 
By Theorem \ref{thm:SDR}, it is enough to show that $|B|$ is movable. 
Let $M$ be the movable part of $|B|$ and $N$ the fixed part. 
Then we have $9=H.B>H.N$. 
Due to Lemma \ref{lem:1} and (iv), there exists $a,b \in \Z_{\geq 0}$ such that $B=M+aC_{2}+bC_{5}$. 
%The intersection number of $B,C_{2},C_{5}$ as 
%\[\begin{pmatrix}
%&B&C_{2}&C_{5} \\
%B &6&6&9 \\
%C_{2}&6&-2&4 \\
%C_{5}&9&4&-2
%\end{pmatrix}\]
Thus we have $M^{2}=6-2a^{2}-2b^{2}-12a-18b+8ab \geq 0$. 
It is easy to find that $(a,b)=(0,0)$ is the only non-negative integer solution of this inequality. 
Therefore, $|B|=M$ is movable.  
%By an adjunction formula, the genus of a non-singular member of $|B|$ is $4$. 
\item $3H-\G-B=C_{2}+C_{5}$ is nef since $C_{2}.C_{5}=4$. \qedhere
\end{enumerate}
\end{proof}

By Lemma \ref{lem:Biii2}~(1), $S$ is embedded in $\Q^{3}$ as an anti-canonical member. 
Let $\s \colon W = \Bl_{\G}\Q^{3} \to \Q^{3}$ be the blowing-up of $\Q^{3}$ along $\G$ and $G:=\Exc(\s)$ a $\s$-exceptional divisor. 
As we have seen, there exists a quadric fibration $\pi \colon W \to \P^{1}$ given by $F_{W}:=\s^{\ast}\mc O_{\Q^{3}}(1)-G$. 

In what follows, we regard divisors of $S_{W}$ as those of $S$ because of $S_{W} \simeq S$. 
Then we have $\mc O_{S_{W}}(F_{W})=H-\G$ and hence $(F_{W}.B)_{W}=((H-\G).B)_{S}=3$. 
Let $Z:=\Bl_{B_{W}}W$ denotes the blowing-up of $W$ along $B_{W}$. 
Then we have $(-K_{Z})^{3}=4$ by a straightforward calculation. 
By Lemma \ref{lem:Biii2}~(3), $-K_{Z}$ is nef. 
Therefore, the pair $(\pi \colon W \to \P^{1},B)$ satisfies the condition $(\dag_{6})$ in Proposition \ref{prop:1}. 
Thus we obtain a del Pezzo fibration $\vp \colon X \to \P^{1}$ of degree 6 satisfying the following: 
\[(-K_{X})^{3}=6,\quad -K_{X}.C=0 \text{ and } D \equiv \frac{1}{2}(-K_{Y})-\frac{1}{2}E_{Y}+\frac{3}{2}F_{Y}.\]
Since $Y$ is almost Fano and $-K_{X}.C=0$, $X$ is almost Fano.
By Lemma \ref{lem:3} and \ref{lem:4}, we have $h^{1,2}(X)=4$. 
%$h^{1,2}(X)=h^{1,2}(W)+g_{B}-g_{C}=h^{1,2}(\Q^{3})+g_{\G}+g_{B}-g_{C}=4$. 
Therefore, $X$ does not belong to Case (B-i-2). Thus $X$ belongs to Case (B-iii-2). \bbox
\subsection{Case (B-iii-3)}\label{ss:Biii3}
In order to construct an example belonging to Case (B-iii-3), we consider the following lattice:
\[\begin{pmatrix}
&H_{\a}&H_{\b}&F&B \\
H_{\a}&0&2&2&0 \\
H_{\b}&2&0&2&3 \\
F&2&2&0&3 \\
B &0&3&3&-2 
\end{pmatrix}.\] 

This is an even lattice with signature $(1,3)$ and hence there exists a K3 surface $S$ such that $\Pic(S)$ is isometry to this lattice by Theorem \ref{thm:NSL}. Moreover, we may assume that $H:=H_{\a}+H_{\b}$ is nef and big. 

Set $C_{1}:=2H_{\a}-B$, $C_{1}':=-H_{\a}+B$, $C_{3}:=B$ and $C_{3}':=3H_{\a}-B$. Then we have $C_{i},C_{i}'$ are $(-2)$-curves with $H.C_{i}=H.C_{i}'=i$ for $i \in \{1,3\}$. 

Let $C=xH_{\a}+yH_{\b}+zF+wB$ be a divisor on $S$ with $x,y,z,w \in \Z$. Then we have 
\[C^{2}=-2w^{2}+4xy+4xz+4yz+6yw+6zw \text{ and } H.C=2x+2y+4z+3w. \]
By solving these for $x$ and $y$, we have the following: 
\begin{align*}
x&=\frac{(H.C)-4z-6w \pm \sqrt{(H.C)^{2}-4C^{2}-8w^{2}-16z^{2}}}{4} \text{ and } \\
y&=\frac{(H.C)-4z \mp \sqrt{(H.C)^{2}-4C^{2}-8w^{2}-16z^{2}}}{4}
\end{align*}
Hence the following (i) and (ii) holds by a straightforward calculation.
\begin{itemize}
\item[(i)] It is impossible that $C^{2}=0$ and $H.C=1$.
\item[(ii)] If $C^{2}=-2$ and $H.C \leq 4$, then $C=C_{1}$, $C_{1}'$, $C_{3}$ or $C_{3}'$. 
\end{itemize}
\begin{lem}\label{lem:Biii3}
\begin{enumerate}
\item $H$ is base point free and ample but not very ample. 
\item $F$ is base point free. 
\item $H+F$ is very ample. 
\item Let $f \colon S \to \P^{3}$ be a morphism given by $|H|$. Then the image $f(S)$ is non-singular quadric surface. 
\item Let $p \colon S \to \P^{1}$ be a morphism given by $|F|$. 
Then the composition of $f \times p \colon S \to \P^{3} \times \P^{1}$ and the segre embedding $\P^{3} \times \P^{1} \to \P^{7}$ is exactly given by $|H+F|$. 
In particular, $f \times p \colon S \to \P^{3} \times \P^{1}$ is closed embedding. 
\item There exists non-singular quadric $Q \sb \P^{3}$ such that $S \sb Q \times \P^{1} \sb \P^{3} \times \P^{1}$. 
\item $2H-B$ is nef. 
\end{enumerate}
\end{lem}
\begin{proof}
\begin{enumerate}
\item Since (i),(ii) and Theorem \ref{thm:SDR}, $H$ is base point free and ample. 
%If there exists an integral curve $C$ such that $H.C=0$, we have $C^{2}<0$ by the Hodge index theorem and hence $C^{2}=-2$. 
%But there exists no such curve by (ii) and hence $H$ is ample. 
Since $H.H_{\a}=2$ and $H_{\a}^{2}=0$, $H$ is not very ample by Theorem \ref{thm:SDR}. 
\item %We can prove this in similar manners as Lemma \ref{lem:Biii2}~(2). 
It is enough to show that $|F|$ is movable. 
Let $M$ be the movable part of $|F|$ and $N$ the fixed part. 
Then we have $4=H.F>H.N$. 
By Lemma \ref{lem:1} and (iv), there exists $a,b,c,d \in \Z_{\geq 0}$ such that $F=M+aC_{1}+bC_{1}'+cC_{3}+dC_{3}'$. 
Since $H.M > 0$ and $M^{2}>0$, we have $a=b=0$ and hence $F=M$ is movable. 
%$H.M=4-(a+b+3c+3d)>0$. 
%If $d=1$, then $M=F-(3H_{\a}-B)$ and $M^{2}=F^{2}-2F.(3H_{\a}-B)+(3H_{\a}-B)^{2}=0-2 \times 3 + (-2)=-8$, which is contradiction. 
%If $c=1$, then $M=F-B$ and $M^{2}=0-6-2=-8$, which is contradiction. 
%Thus $c=d=0$ and hence we have $M^{2}=-2a^{2}-2b^{2}-2a-2b+4ab \geq 0$. 
%It is easy to see that $(a,b)=(0,0)$. 
\item By using Theorem \ref{thm:SDR}, the very ampleness of $H+F$ can be obtained by the same manner as we have seen. 
%Let us take $C$ be a divisor on $S$ and describe $C$ as $C=xH_{\a}+yH_{\b}+zF+wB$ where $x,y,z,w \in \Z_{\geq 0}$. 
%\begin{itemize}
%\item Assume $C^{2}=0$ and $(H+F).C=1$. Since $H$ is ample and $F$ is nef, we have $H.C=1$ and $F.C=0$. But it contradict $(i)$. 
%\item It is impossible $C^{2}=-2$ and $(H+F).C=0$ because of the ampleness of $H$. 
%\item Assume $C^{2}=0$ and $(H+F).C=2$. Then we have $H.C=2$ and $F.C=0$. By solving this we have $y=\pm \frac{\sqrt{-1}w}{\sqrt{2}}$ which is impossible. 
%\end{itemize}
%Hence $(H+F)$ is very ample by Theorem \ref{thm:SDR}. 
\item $f \colon S \to \P^{3}$ is not closed embedding but finite morphism onto the image. 
Since $f$ is given by the complete linear system $|H|$, $Q:=f(S)$ is an irreducible and reduced quadric surface and $S \to Q$ is $2:1$. 
Assume $Q$ is a singular cone with a vertex $o \in Q$. 
Then a line $l$ of $Q$ is a non-Cartier Weil divisor and the closure $\ol{f^{-1}(l \setminus \{o\})}=D$ is a Cartier divisor of $S$. 
Since $f^{\ast}(2l)=H$, we have $H=2D$ and it contradicts that $H$ is a generator of $\Pic(S)$. Thus $Q$ is non-singular. 
\item It is clear. 
\item It is clear from (4) and (5). 
\item Let us prove the nefness of $2H-B$. 
Since $H.(2H-B)=5$ and $(2H-B)^{2}=2$, we have $2H-B$ is effective. 
%$(2H-B)^{2}=4 \times 4-4 \times 3 -2=2$
Let $M$ be the movable part of $|2H-B|$ and $N$ the fixed part. 
Since $M \neq 0$ and $H$ is ample, we have $H.N < 5$. 
By Lemma \ref{lem:1} and (iv), there exists $a,b,c,d \in \Z_{\geq 0}$ such that $F=M+aC_{1}+bC_{1}'+cC_{3}+dC_{3}'$. 
Since $(2H-B).C_{1}=0$, $(2H-B).C_{1}'=4$, $(2H-B).C_{3}=8$ and $(2H-B).C_{3}'=4$, we have that $2H-B$ is nef. \qedhere
\end{enumerate}
\end{proof}
\begin{lem}
Set $H_{\xi}=\mc O_{\P^{3} \times \P^{1}}(1,0)$ and $H_{\eta}=\mc O_{\P^{3} \times \P^{1}}(0,1)$. Then there exists a non-singular member $W \in |2H_{\xi} + 2H_{\eta}|$ which contains $S$. 
\end{lem}
\begin{proof}
Consider the following exact sequence:
\[0 \to \mc I_{S} \otimes \mc O(2H_{\xi}+2H_{\eta}) \to \mc O(2H_{\xi}+2H_{\eta}) \to \mc O(2H_{\xi}+2H_{\eta})|_{S} \to 0.\]
We have $h^{0}(\P^{3} \times \P^{1},\mc O(2H_{\xi}+2H_{\eta}))=30$ and $h^{0}(S,\mc O(2H_{\xi}+2H_{\eta})|_{S})=26$ and hence there exists a member of $|2H_{\xi}+2H_{\eta}|$ which contains $S$. 
%$h^{0}(S,\mc O(2H_{\xi}+2H_{\eta})|_{S})=\frac{1}{2} \cdot 4(H_{\xi}^{2}+2H_{\xi}H_{\eta}+H_{\eta}^{2})+2=2 \cdot (4+2 \cdot 4)+2=26$
Set $V:=H^{0}(\P^{3} \times \P^{1},\mc O(2H_{\xi}+2H_{\eta}) \otimes \mc I_{S})$ and let $\L:=|V|$ be the linear system corresponding to $V$. 

Let us prove that $\Bs \L=S$. 
By Lemma \ref{lem:Biii3} (6), there exists a non-singular quadric surface $Q \sb \P^{3}$ such that $S \sb Q \times \P^{1}$. 
Set $R:=Q \times \P^{1}$. Then $R$ is linear equivalent to $2H_{\xi}$ as a divisor of $\P^{3} \times \P^{1}$. 
Hence we have $R+2H_{\eta} \in \L$ and $\Bs \L \sb R$. 
Thus $\Bs \L=\Bs \L \cap R=\Bs \L|_{R}$ holds. 
Let us fix an isomorphism $R \simeq \P^{1}_{\a} \times \P^{1}_{\b} \times \P^{1}_{\eta}$ and 
define $H_{\a}:=\mc O(1,0,0)$ and $H_{\b}:=\mc O(0,1,0)$. Note that $H_{\eta}|_{R}=\mc O(0,0,1)$. 
Now $\L|_{R}$ is the linear system corresponding to the linear space $V_{R}=H^{0}(\P^{1}_{\a} \times \P^{1}_{\b} \times \P^{1}_{\g} , \mc O(2H_{\a}+2H_{\b}+2H_{\eta}) \otimes \mc I_{S/R})$. 
Since $\mc I_{S}=\mc O(-2H_{\a}-2H_{\b}-2H_{\eta})$, we have $\dim V_{R}=1$ and hence $|\L|_{R}|=\{S\}$. Thus we have that $\Bs \L=S$. 

Therefore, a general member $W$ of $\L$ have singularities only at $S$. 
Since $\L$ is movable and $\dim \L \geq 3$, a general member $W$ is irreducible and reduced by Bertini's theorem. 
Moreover, $S=W \cap R$ is a non-singular Cartier divisor of $W$ and hence $W$ is non-singular along $S$. By this argument, general members of $\L$ are non-singular. 
\end{proof}

Let $W$ be a general member of $\L$ and $\pi \colon W \to \P^{1}$ the restriction to W of the projection $\P^{3} \times \P^{1} \to \P^{1}$. 
Thus $\pi$ gives a quadric fibration structure and $\deg (\pi|_{B})=3$. 
Let $Z=\Bl_{B_{W}}W$ denotes the blowing-up of $W$ along $B_{W}$. 
It is easy to see that $(-K_{Z})^{3}=2$. 
By Lemma \ref{lem:Biii3} (7), $-K_{Z}$ is nef and big. 
Therefore, the pair $(\pi \colon W \to \P^{1},B)$ satisfies the condition $(\dag_{6})$ in Proposition \ref{prop:1}.
Thus we have a del Pezzo fibration of degree 6 $\vp \colon X \to \P^{1}$ satisfying the following: 
\[(-K_{X})^{3}=4,\quad -K_{X}.C=0 \text{ and } D\equiv\frac{1}{2}(-K_{Y})-\frac{1}{2}E_{Y}+F.\]
Since $Y$ is almost Fano and $-K_{X}.C=0$, $X$ is almost Fano and belongs to Case (B-i-3) or Case (B-iii-3). 
\begin{lem}
$h^{1,2}(X)=3$．
\end{lem}
\begin{proof}
In the diagram $(\bigstar_{6})$ for this case, $Y \to X$ is the blowing-up along $B \simeq \P^{1}$ and $Y \dra Z$ is flop and $Z \to W$ is the blowing-up along $C \simeq \P^{1}$. 
By Lemma \ref{lem:3} and \ref{lem:4}, we have $h^{1,2}(X)=h^{1,2}(W)$. Note that $W$ is a non-singular member of $|2H_{\xi}+2H_{\eta}|$ in $\P^{3}_{\xi} \times \P^{1}_{\eta}$. 

Consider a $4 \times 4$ symmetric matrix $M(\eta_{0},\eta_{1})=(f_{ij}(\eta_{0},\eta_{1}))_{0 \leq i,j \leq 3}$ over $\C[\eta_{0},\eta_{1}]$ 
such that each $f_{ij}(\eta_{0},\eta_{1})$ is a homogenous polynomial of degree 2. Let 
$D=\left(\sum f_{ij}(\eta_{0},\eta_{1})\xi_{i}\xi_{j}=0\right)$ 
be a member of $|2H_{\xi}+2H_{\eta}|$. 
By taking general $M(\eta_{0},\eta_{1})$, we may assume that $D$ is non-singular and $\det M(\eta_{0},\eta_{1})$ has no multiple root. 
Thus the singular fibers of $D \to \P^{1}$ is determined by $\det M(\eta_{0},\eta_{1})=0$. In particular, the number of the singular fibers of $D \to \P^{1}$ is 8. 
Hence we have 
\[\eu(D)=\eu(\P^{1} \times \P^{1}) \cdot \eu(\P^{1}) + 8 \left( \eu(\Q^{2}_{0})-\eu(\P^{1} \times \P^{1}) \right)=0, \]
where $\eu$ denotes the topological Euler number and $\Q^{2}_{0}=\{xy-z^{2}=0\} \sb \P^{3}$ denotes the quadric cone. 
On the other hand, we have $b_{0}(D)=b_{6}(D)=1$, $b_{1}(D)=b_{5}(D)=0$ and $b_{2}(D)=b_{4}(D)=2$, where $b_{i}$ denotes the $i$-th Betti number. 
Thus we have $b_{3}(D)=6$. 
%$h^{1,1}(D)=h^{2,2}(D)=2$ and $h^{0,0}(D)=h^{3,3}(D)=1$ and other Hodge numbers are equal to $0$ except $h^{1,2}(D)=h^{2,1}(D)$. Since $\eu(D)=6-2h^{1,2}(D)$, we obtain $h^{1,2}(D)=3$. 
Due to Bertini's theorem, 
$\{D \in |2H+2F| \mid D \text{ is non-singular }\}$ is an open set of $|2H+2F|$ and connected in the Euclidean topology. In particular, $D$ is deformation equivalent to $W$. Thus we have $b_{3}(D)=b_{3}(W)$. 
Since $h^{0,3}(W)=0$, we obtain $h^{1,2}(W)=3$. 
\end{proof}
By this lemma, $X$ can not belong to Case (B-i-3). Therefore, $X$ belongs to Case (B-iii-3). \bbox
\subsection{Case (B-iii-4)}\label{ss:Biii4}
In order to construct an example belonging to Case (B-iii-4), 
we consider the following lattice: 
\[\begin{pmatrix}
&H&\G&B \\
H&6&2&11 \\
\G&2&-2&8 \\
B&11&8&8
\end{pmatrix}.\]
This is an even lattice with signature $(1,2)$. By virtue of Theorem \ref{thm:NSL}, there exists a K3 surface $S$ such that $\Pic(S)$ is isometry to this lattice. Moreover, we may assume that $H$ is nef and big. 

Set $C_{2}:=\G$, $C_{5}:=3H-\G-B$ and $C_{7}:=4H-3\G-B$. 
Note that $C_{i}$ is $(-2)$-curve with $H.C_{i}=i$ for $i \in \{2,5,7\}$. 

Let $C=xH+y\G+zB$ be a divisor on $S$ with $x,y,z \in \Z$. Then we have 
\[C^{2}=6x^{2}-2y^{2}+8z^{2}+4xy+22xz+16yz \text{ and } H.C=6x+2y+11z.\]
By solving these equations, we have the following: 
\begin{align*}
&x=\frac{4l-57z \pm \sqrt{4(H.C)^{2}-24C^{2}-123z^{2}}}{24} \text{ and } \\
&y=\frac{13z \mp \sqrt{4(H.C)^{2}-24C^{2}-123z^{2}}}{8}.
\end{align*}
By a straightforward calculation, we can prove that the following (i)-(iv). 
\begin{itemize}
\item[(i)] It is impossible that $C^{2}=0$ and $H.C=1$. 
\item[(ii)] It is impossible that $C^{2}=0$ and $H.C=2$. 
\item[(iii)] It is impossible that $C^{2}=0$ and $H.C=3$.
\item[(iv)] If $C^{2}=-2$ and $H.C \leq 10$, then $C=C_{2}$, $C_{5}$ or $C_{7}$. 
\end{itemize}
\begin{lem}
\begin{enumerate}
\item $H$ is very ample and $S$ is embedded in $\Q^{3}$ as an anti-canonical divisor. 
\item There exists a non-singular member of $|B|$ meeting $\G$ transversally. 
\end{enumerate}
\end{lem}
\begin{proof}
\begin{enumerate}
\item It can be proved similarly to Lemma \ref{lem:Bii}~(1). 
%By (i),(ii),(iv) and Theorem \ref{thm:SDR}, $H$ is very ample. By (iii) and Lemma \ref{lem:2}, $S$ is embedded in $\Q^{3}$ as an anti-canonical divisor. 
\item It is enough to show that $|B|$ is movable. 
Let $M$ denotes the movable part of $|B|$ and $N$ the fixed part. 
Then we have $11=H.B>H.N$ 
By Lemma \ref{lem:1} and (iv), there exists $a,b,c \in \Z_{\geq 0}$ such that $B=M+a\G+bC_{5}+cC_{7}$. Since $H$ is very ample and $M$ is movable, we have the following inequalities:
\begin{align*}
&11=H.B=(H.M)+2a+5b+7c>2a+5b+7c \text{ and } \\
&M^{2}=8-(18a+36b+26c)+12ac+6bc \geq 0. 
\end{align*}
It is easy to see that $a=b=c=0$ is the only non-negative integer solution of these inequalities. 
%Actually, $c$ equals to $0$ or $1$. 
%If $c=1$, then $a=1$, $b=0$ and $M^{2}=8-44+12<0$ which is impossible. 
%If $c=0$, then $M^{2}=8-(18a+36b)$ and hence we obtain $a=b=0$. 
Hence $|B|=M$ is movable. \qedhere
\end{enumerate}
\end{proof}

Hence there exists an embedding $S \hra \Q^{3}$. 
Let $\s \colon W:=\Bl_{\G}\Q^{3} \to \Q^{3}$ be the blowing-up of $W$ along $\G$ and $G=\Exc(\s)$. As we have seen, there exists a quadric fibration 
$\pi \colon W \to \P^{1}$ given by $F_{W}:=\s^{\ast}\mc O_{\Q^{3}}(1)-G$. 
Let us identify $S_{W},B_{W}$ with $S,B$ respectively and regard divisors of $S_{W}$ as those of $S$. 
Then $\mc O_{S_{W}}(F_{W})=H-\G$ and hence $(F_{W}.B)_{W}=((H-\G).B)_{S}=3$. 
Set $\t \colon Z:=\Bl_{B}W \to W$ and $p:=\pi \circ \t \colon Z \to \P^{1}$. 
\begin{lem}
$-K_{Z}$ is $p$-nef and $p$-big. 
\end{lem}
\begin{proof}
It is clear that $-K_{Z}$ is $p$-big. 
Let $\g \sb Z$ be a curve which is contained in some $p$-fiber. 
If $-K_{Z}.\g<0$, then $S_{Z} \simeq S$ contains $\g$. 
Since $-K_{Z}|_{S_{Z}}=3H-\G-B=C_{5}$, we have $\g=C_{5}$. 
%which means that $S_{Z}$ contains $\g$. 
%Since $S_{Z}$ is isomorphic to $S$, we can regard divisors of $S_{Z}$ as those of $S$. 
Since $F_{Z}|_{S_{Z}}=H-\G$, we have $F_{Z}.\g=(H-\G).C_{5}=5$, which contradicts that $\g$ is contracted by $p$. 
Hence $-K_{Z}$ is $p$-nef. 
\end{proof}
Therefore, the pair $(\pi \colon W \to \P^{1},B)$ satisfies the condition $(\dag_{6})$. 
By Proposition \ref{prop:1}, we have the diagram $(\bigstar_{6})$. 
In particular, we obtain a del Pezzo fibration of degree 6 $\vp \colon X \to \P^{1}$ satisfying the following: 
\[(-K_{X})^{3}=2,\quad -K_{X}.C=1 \text{ and } D \equiv \frac{1}{2}(-K_{Y})-\frac{1}{2}E_{Y}+\frac{5}{2}F.\]
\begin{lem}\label{lem:almFBiii4}
$-K_{X}$ is nef and big but not ample. 
\end{lem}
\begin{proof}
By Mori-Mukai's classification, $X$ is not Fano. 
Hence it is enough to show the nefness of $-K_{X}$. 

At first, we prove that $\Phi|_{S_{Z}} \colon S_{Z} \dra S_{Y}$ is an isomorphism. 
Let $\g$ be an irreducible flopping curve of $\Phi$ and assume $S_{Z} \simeq S$ contains $\g$ as a divisor. 
Thus there exists $x,y,z \in \Z$ such that $\g=xH+y\G+zB$ holds in $\Pic(S)$. 
Note that $-K_{Z}|_{S_{Z}}=C_{5}$, $F_{Z}|_{S_{Z}}=H-\G$ and $\g$ is $(-2)$-curve. 
Hence the following equations hold: 
\[(3H-\G-B).\g=(H-\G).\g=0 \text{ and }\g^{2}=-2.\]
But these equations have no integer solution for $x,y,z$. 
Hence $S_{Z}$ has no flopping curve of $\Phi$. 
Thus $\Phi|_{S_{Z}} \colon S_{Z} \dra S_{Y}$ is an isomorphism. 

Let us prove the nefness of $-K_{X}$. 
Let $\g \sb X$ be a curve satisfing $(-K_{X}).\g<0$. 
Then we have $\g \neq C=\mu(D)$ and hence $0>\mu^{\ast}(-K_{X}).\g=(-K_{Y}+D).\g_{Y}$, where $D=\Exc(\mu)$. 
Since $\g \neq C$, $D$ does not contain $\g_{Y}$ and hence we have $D.\g_{Y} \geq 0$. 
Thus $-K_{Y}.\g_{Y} < 0$ which means $S_{Y}$ contains $\g_{Y}$. 
Since $-K_{Y}|_{S_{Y}}=3H-\G-B=C_{5}$, we have $\g_{Y}=C_{5}$. 
Now, $D|_{S_{Y}}=4H-3\G-B=C_{7}$ and hence we have $(-K_{Y}+D)|_{S_{Y}}=C_{5}+C_{7}$. Therefore, we have $(-K_{Y}+D).\g_{Y}=(C_{5}+C_{7}).C_{5}=-2+3=1$ and it is contradiction. Thus $-K_{X}$ is nef. 
\end{proof}
By Lemma \ref{lem:almFBiii4}, $\vp \colon X \to \P^{1}$ is an almost Fano del Pezzo fibration of degree $6$ with $(-K_{X})^{3}=2$. Hence $X$ belongs to Case (B-iii-4). 
By Lemma \ref{lem:3} and \ref{lem:4}, we have 
$h^{1,2}(X)=5$. 
%h^{1,2}(Y)=h^{1,2}(Z)=h^{1,2}(W)+g_{B}=h^{1,2}(\Q^{3})+g_{B}+g_{\G}=5$.
\bbox
\subsection{$\P^{2}$-bundles with quinque-sections}
In this subsection, we state and prove Proposition \ref{prop:2}, which is similar to Proposition \ref{prop:1}, to construct an example belonging to Case (A-2). 
\begin{prop}\label{prop:2}
Let $\pi \colon W \to \P^{1}$ be a $\P^{2}$-bundle and 
$B \sb W$ a smooth projective curve and 
$\t \colon Z:=\Bl_{B}W \to W$ the blowing-up along $B$. 
We assume the following condition $(\dag_{5})$ for a pair $(\pi \colon W \to \P^{1},B)$.
\[(\dag_{5}) \cdots
\left\{ \begin{array}{ll}
\deg (\pi|_{B} \colon B \to \P^{1})=5 \text{ and } \\
-K_{Z} \text{ is $p$-nef and $p$-big with } p := \pi \circ \t \colon Z \to \P^{1}. 
\end{array} \right. \]
Then the following holds.
\begin{enumerate}
\item The morphism from $Z$ to the relative anti-canonical model over $\P^{1}$ 
\[ \Psi \colon Z \to \ol{Z}:=\Proj \bigoplus_{n \geq 0} p_{\ast}\mc O_{Z}(-nK_{Z})\]
is a small contraction or an isomorphism. 
\item Define a variety $Y$ and a birational map $\Phi \colon Z \dra Y$ as follows: 
\begin{itemize}
\item If $-K_{Z}$ is $p$-ample, set $Y:=Z$ and $\Phi:=\id_{Z} \colon Z \to Y$; or 
\item if $-K_{Z}$ is not $p$-ample, let $\Phi \colon Z \dra Y$ be the flop over $\P^{1}$ of $\Psi \colon Z \to \ol{Z}$. 
\end{itemize}

Let $q \colon Y \to \P^{1}$ be a structure morphism onto $\P^{1}$ and $\mu \colon Y \to X$ a $K_{Y}$-negative ray contraction over $\P^{1}$. If $-K_{Z}$ is $p$-ample, we choose the other ray which is different to the one corresponding to $\t \colon Z \to W$. Let $\vp \colon X \to \P^{1}$ denotes the structure morphism onto $\P^{1}$. 

Then $\mu$ is the blowing-up along a $\vp$-section $C$ and $\vp \colon X \to \P^{1}$ is a del Pezzo fibration of degree 5. 
\[\xymatrix{
&D \ar[ld]& \ar@{}[l]|{\sb} \ar[ld]_{\mu}Y=\Bl_{C}X \ar[dd]^{q}\ar@{<--}[r]^{\Phi}& Z=\Bl_{B}W \ar@{}[r]|{\sp} \ar[rd]^{\t} \ar[dd]_{p}& E \ar[rd]& \\
C&\ar@{}[l]|{\sb} X \ar[rd]_{\vp}& && W \ar@{}[r]|{\sp} \ar[ld]^{\pi} &B&(\bigstar_{5}) \\
&&\P^{1} \ar@{=}[r]&\P^{1}&&
} \]
\item Let $F_{Y}$ be a general $q$-fiber and we set $D:=\Exc(\mu)$, $E:=\Exc(\t)$ and $E_{Y}:=\Phi_{\ast}E$. Then we obtain 
\[D \equiv \frac{2}{3}(-K_{Y})-\frac{1}{3}E_{Y}+zF_{Y}.\]
Moreover, the following equalities hold: 
\begin{align*}
(-K_{X})^{3}&=22-2g_{B}, \\
-K_{X}.C&=(-K_{W}).B-2g_{B}-16 \text{ and } \\
z&=\frac{2}{3}(-K_{W}).B-g_{B}-12.
\end{align*}
\end{enumerate}
\end{prop}
\begin{proof}
We can prove this proposition similarly to Proposition \ref{prop:2}. 
\begin{enumerate}
\item Let us assume that $\Psi \colon Z \to \ol{Z}$ is a divisorial contraction and set $D=\Exc(\Psi)$. 
Let $F_{Z}$ be a general $p$-fiber and $F_{W}$ a general $\pi$-fiber such that $\s(F_{Z})=F_{W}$. 
Now $\s|_{F_{Z}} \colon F_{Z} \to F_{W} \simeq \P^{2}$ is the blowing-up at reduced five points. 

Let us prove that $D|_{F_{Z}}$ is a disjoint union of $(-2)$-curves $l_{1},\ldots,l_{n}$ and $n \in \{1,2\}$. 
Let $e_{1},e_{2},e_{3},e_{4},e_{5}$ be exceptional curves of $\t|_{F_{Z}}$ and set $h:=(\t|_{F_{Z}})^{\ast}\mc O_{\P^{2}}(1)$. 
For every irreducible curve $l \sb F_{Z}$, 
there exists $a,b_{j} \in \Z_{\geq 0}$ such that $l=ah-\sum_{j=1}^{5} b_{j}e_{j}$. 
If $l$ is contracted by $\t|_{F_{Z}}$, then $l$ is $(-2)$-curve and we have the following: 
\[0=3a-\sum_{j=1}^{5}b_{j} \text{ and } -2=a^{2}-\sum_{j=1}^{5}b_{j}^{2}.\]
The following inequality 
\[a^{2}+2=\sum_{j=1}^{5}b_{j}^{2} \geq 5 \left( \frac{\sum_{j=1}^{5} b_{j}}{5} \right)=\frac{9}{5}a^{2}\]
implies that $a=1$. Thus three of the five numbers $b_{1},\ldots,b_{5}$ are 1 and the remaining two are 0. 
Hence $l$ is the proper transform of a line passing through exactly three points of the five points. 
By our assumption, $-K_{F_{Z}}$ is nef and hence any four points are not collinear. 
Therefore, the number of lines passing through exactly three points of the five points is at most two. 
Moreover, if there exists two lines passing through three points of the five points, then the proper transforms of those do not meet since the two lines meet transversally at one point of the five points.

There exists $x,y,z \in \Q$ such that $D \equiv x(-K_{Z})+yE+zF_{Z}$. Then we have the following equations:
%Since $D|_{F_{Z}}=x(-K_{F_{Z}})+y(E|_{F_{Z}})$, 
\begin{align*}
0&=(-K_{F_{Z}})(D|_{F_{Z}})=4x+5y \text{ and } \\
-2n&=(D|_{F_{Z}})^{2}=4x^{2}-5y^{2}+10xy. 
\end{align*}
By these equations, we have $x=\pm \sqrt{\frac{5n}{18}}$. This contradicts $x \in \Q$ and $n \in \{1,2\}$. 
Therefore, $\Psi$ is an isomorphism or a small contraction. 
\item %Set $Y$ and $\Phi \colon Z \dra Y$ as in the statement. For $q \colon Y \to \P^{1}$, 
%For a general $p$-fiber $F_{Z}$ and the proper transformation $F_{Y}$ of $F_{Z}$, the birational map $\Phi|_{F_{Z}} \colon F_{Z} \dra F_{Y}$ is an isomorphism. We abbreviate $F_{Y},F_{Z}$ to $F$ for simplicity. 
Let us prove that $\mu \colon Y \to X$ is the blowing-up along a non-singular curve $C$. 
Since $\rho(Y)=3$, we have $\dim X \geq 2$. 
If $\dim X=2$, then the following equations holds by the same argument in the proof of Proposition \ref{prop:1} : 
\begin{align*}
2=4x+5y \text{ and } 0=4x^{2}-5y^{2}+10xy \text{ with } x,y\in \Q. 
\end{align*}
%$X$ is isomorphic to a Hirzebruch surface $\F_{n}$ by Lemma \ref{lem:5}.Set $D=g^{\ast}\mc O_{\F_{n}}(1)$ and denote $D=x(-K_{Z})+yE+zF_{Z}$ as a $\Q$-divisor. Then we have the following equations: 
%\begin{align*}
%2&=(-K_{Y}).D.F=(-K_{F})D|_{F}=4x+5y \text{ and } \\
%0&=D^{2}F=(D|_{F})^{2}=4x^{2}-5y^{2}+10xy
%\end{align*}
Hence $x=\frac{3 \pm \sqrt{5}}{6}$, $y=\mp \frac{2}{3\sqrt{5}}$ and it contradicts $x,y \in \Q$. 
Hence $\mu$ is a divisorial contraction. 
By the similar arguments in the proof of Proposition \ref{prop:1}, we can prove that $\mu(D)$ is not a point. Hence we have that $\mu(D)=:C$ is a smooth curve, where $D=\Exc(\mu)$. 

Denote $D \equiv x(-K_{Z})+yE+zF_{Z}$ with $x,y,z \in \Q$ and $m:=F_{X}.C \in \Z_{\geq 0}$. 
Since $(-K_{F_{Y}})^{2}=4$ and $(-K_{F_{X}})^{2} \leq 9$, we have $0 \leq m \leq 5$. 
By a calculation of some intersection numbers, we have the following: 
\begin{align*}
m&=(-K_{Y})DF=4x+5y, \\
-m&=D^{2}F=4x^{2}-5y^{2}+10xy \text{ and } \\
\Z &\ni 5x-5y=EFD. 
\end{align*}
Hence we have the following:
\begin{align*}
&x=\frac{3m \pm \sqrt{5m^{2}+20m}}{12},y=\mp \sqrt{\frac{m^{2}+4m}{45}} \text{ and } \\
&5x-5y=\frac{5m \pm 3 \sqrt{5m^{2}+20m}}{4} \in \Z.
\end{align*}
By using these relations, the possibilities of $(m,x,y)$ are as follows:
\[(m,x,y)=(0,0,0),\left(1,\frac{2}{3},-\frac{1}{3} \right), (5,0,1)\]
%If $x=y=0$, then $D$ or $-D$ is nef which is impossible. Also 
We can see that $(x,y) \neq (0,0),(0,1)$ by the same argument in the proof of Proposition \ref{prop:1}. 
%If $(m,x,y)=(5,0,1)$, then $D|_{F}=E|_{F}$ and $F_{W}$ is birationally transformed into $F_{W}$ by the birational map $\psi \colon W \dra X$. Since $\psi$ is isomorphic in codimension 1, $\psi$ is an isomorphism by Lemma \ref{lem:6}, and $\psi(C)=B$. In the case that $-K_{Z}$ is $p$-ample, it contradicts that the rays which corresponding $\t$, $\mu$ are different. In the case that $-K_{Z}$ is not $p$-ample, it contradicts that $\Psi \colon Y \dra Z$ is not an isomorphism over $\P^{1}$. 
Hence we have that $m=1$, $x=\frac{2}{3}$, $y=-\frac{1}{3}$ and $(-K_{F_{X}})^{2}=5$. 
Thus $\mu$ is the blowing-up along a $\vp$-section and $\vp$ is a del Pezzo fibration of degree $5$. 
\item Since $x=\frac{2}{3}$, $y=-\frac{1}{3}$, $(-K_{Z})^{3}=(-K_{Y})^{3}$, $(-K_{Z})^{2}D=(-K_{Y})^{2}D_{Y}$ and $(-K_{Z})D^{2}=(-K_{Y})D_{Y}^{2}$, we obtain the following equalities. 
\begin{align*}
&(-K_{X})^{3}-2(-K_{X}).C-2=(-K_{W})^{3}-2(-K_{W}).B+(2g_{B}-2), \\
&(-K_{X}).C+2=\frac{2\left((-K_{W})^{3}-2(-K_{W}).B+(2g_{B}-2) \right)}{3}-\frac{(-K_{W}).B+2-2g_{B}}{3}+4z \text{ and } \\
&-2=\frac{4\left( (-K_{W})^{3}-2(-K_{W}).B+(2g_{B}-2)\right)}{9}+\frac{2g_{B}-2}{9}-\frac{4\left( (-K_{W}).B+2-2g_{B}\right)}{9}+2z.
\end{align*}
By solving these equalities, we have the following. 
\begin{align*}
(-K_{X})^{3}&=\frac{5(-K_{W})^{3}-18g_{B}-72}{9}, \\
-K_{X}.C&=\frac{9 (-K_{W}).B-18g_{B}-2(-K_{W})^{3}-36}{9} \text{ and } \\
z&=\frac{6(-K_{W}).B-9g_{B}-2(-K_{W})^{3}}{9}.
\end{align*}
Note that $(-K_{W})^{3}=54$ since $W \to \P^{1}$ is $\P^{2}$-bundle. 
Thus we have the last statement in Proposition \ref{prop:2}. \qedhere
\end{enumerate}
\end{proof}
\subsection{Case (A-2)}\label{ss:A2}
Let us construct an example belonging to Case (A-2). 

Let $C \sb \P^{2}$ be a smooth elliptic curve and set $S,W$ as follows: 
\[S:=C \times \P^{1} \sb W:=\P^{2} \times \P^{1}.\]
Set $H:=\mc O_{\P^{2} \times \P^{1}}(1,0)$ and $F:=\mc O_{\P^{2} \times \P^{1}}(0,1)$. Note that $-K_{W}=3H+2F$. 

Let $\pr_{1} \colon S \to C$ is the projection of the first factor and $\pr_{2} \colon S \to \P^{1}$ is that of the second factor. 
Let $B \in |\pr_{1}^{\ast}\mc O_{C}( 5p ) \otimes \pr_{2}^{\ast}\mc O_{\P^{1}}(2)|$ be a smooth divisor of $S$, where $p$ is a point of $C$. 
Then we have $g_{B}=6$ by adjunction formula and we have $F.B=5$, $H.B=6$ and $-K_{W}.B=28$. 
Let $\t \colon Z:=\Bl_{B}W \to W$ be the blowing-up of $W=\P^{2} \times \P^{1}$ along $B$ and $E$ the $\t$-exceptional divisor. 
Note that $S_{Z}:=\t^{-1}_{\ast}S$ is linearly equivalent to $\t^{\ast}(3H)-E$ and $\t|_{S_{Z}} \colon S_{Z} \to S$ is isomorphic. 
\begin{lem}
$-K_{Z}$ is nef and big. 
\end{lem}
\begin{proof}
It is easy to find that $(-K_{Z})^{3}=12$. 
Since $-K_{Z}=\t^{\ast}(3H+2F)-E=S_{Z}+2F_{Z}$, it is enough to show that $-K_{Z}|_{S_{Z}}$ is nef. 
Under the identifying $S_{Z}$ with $S$, we have $-K_{Z}|_{S_{Z}}=\pr_{1}^{\ast}\mc O_{C}( \text{4pts.} )$ and hence $-K_{Z}$ is nef. 
\end{proof}
Therefore, the pair $(\pi \colon \P^{2} \times \P^{1} \to \P^{1},B)$ satisfying the condition $(\dag_{5})$ and 
we obtain the diagram $(\bigstar_{5})$ and a del Pezzo fibration $\vp \colon X \to \P^{1}$ of degree 5 satisfying the following:
\begin{align*}
(-K_{X})^{3}=10,\quad -K_{X}.C=0 \text{ and } D \equiv \frac{2}{3}(-K_{Y})-\frac{1}{3}E_{Y}+\frac{2}{3}F_{Y}.
\end{align*}
Since $-K_{Y}$ is nef and big and $-K_{X}.C=0$, $X$ is almost Fano. 
%Note that $-K_{Z}=\t^{\ast}(3H+2F)-E$ and hence $D_{Z}=\t^{\ast}(2H+2F)-E$. In particular, $S_{Z} \neq D_{Z}$. 
According to \cite{JPR05}, \cite{JPR11} and \cite{Take09}, there are two possibilities of a type of the contraction of the $K_{X}$-trivial ray: a divisorial type or a flopping type. 
\begin{lem}
The $K_{X}$-trivial elementary contraction of $X$ is a divisorial contraction. 
\end{lem}
\begin{proof}
For $\Phi$-flopping curve $l \sb Z$, we have $0=-K_{Z}.l=(S_{Z}+2F_{Z}).l=S_{Z}.l$. 
Since every fiber of $p|_{S_{Z}} \colon S_{Z} \to \P^{1}$ is an elliptic curve, $S_{Z}$ has no $\Phi$-flopping curve. Therefore, $\Phi|_{S_{Z}} \colon S_{Z} \dra S_{Y}$ is isomorphic. 

Let $\G$ be a fiber of $S = C\times \P^{1} \to C$. Note that $-K_{Y}.\G=0$ and 
$D|_{S_{Y}}=\pr_{1}^{\ast}\mc O_{C}(q)$ for a point $q$ of $C$. Hence we have $\mu^{\ast}(-K_{X}).\G=0$. 
Therefore, $X$ has infinitely many $K_{X}$-trivial curves. 
%at least one dimensional family of $-K_{X}$-trivial curves. 
\end{proof}
Hence $X$ belongs to Case (A-2). 
By using Lemma \ref{lem:3} and \ref{lem:4}, we have $h^{1,2}(X)=6$. \bbox
\bibliographystyle{abbrv}

\begin{thebibliography}{10}

\bibitem{BHPV}
W.~P. Barth, K.~Hulek, C.~A.~M. Peters, and A.~Van~de Ven.
\newblock {\em Compact complex surfaces}, volume~4 of {\em Ergebnisse der
  Mathematik und ihrer Grenzgebiete. 3. Folge. A Series of Modern Surveys in
  Mathematics [Results in Mathematics and Related Areas. 3rd Series. A Series
  of Modern Surveys in Mathematics]}.
\newblock Springer-Verlag, Berlin, second edition, 2004.

\bibitem{Cor95}
A.~Corti.
\newblock Factoring birational maps of threefolds after {S}arkisov.
\newblock {\em J. Algebraic Geom.}, 4(2):223--254, 1995.

\bibitem{CM13}
J.~W. Cutrone and N.~A. Marshburn.
\newblock Towards the classification of weak {F}ano threefolds with {$\rho=2$}.
\newblock {\em Cent. Eur. J. Math.}, 11(9):1552--1576, 2013.

\bibitem{Fuj80}
T.~Fujita.
\newblock On the structure of polarized manifolds with total deficiency one.
  {I}.
\newblock {\em J. Math. Soc. Japan}, 32(4):709--725, 1980.

\bibitem{Fuj81}
T.~Fujita.
\newblock On the structure of polarized manifolds with total deficiency one.
  {II}.
\newblock {\em J. Math. Soc. Japan}, 33(3):415--434, 1981.

\bibitem{Fuj84}
T.~Fujita.
\newblock On the structure of polarized manifolds with total deficiency one.
  {III}.
\newblock {\em J. Math. Soc. Japan}, 36(1):75--89, 1984.

\bibitem{HarBook}
R.~Hartshorne.
\newblock {\em Algebraic geometry}.
\newblock Springer-Verlag, New York-Heidelberg, 1977.
\newblock Graduate Texts in Mathematics, No. 52.

\bibitem{Isk77}
V.~A. Iskovskih.
\newblock Fano threefolds. {I}.
\newblock {\em Izv. Akad. Nauk SSSR Ser. Mat.}, 41(3):516--562, 717, 1977.
\newblock English transl. in Math. USSR-Izv. 11 (1977), no. 3, 485–527
  (1978).

\bibitem{Isk78}
V.~A. Iskovskih.
\newblock Fano threefolds. {II}.
\newblock {\em Izv. Akad. Nauk SSSR Ser. Mat.}, 42(3):506--549, 1978.
\newblock English transl. in Math. USSR-Izv. 12 (1978), no. 3, 469–506
  (1979).

\bibitem{JPR05}
P.~Jahnke, T.~Peternell, and I.~Radloff.
\newblock Threefolds with big and nef anticanonical bundles. {I}.
\newblock {\em Math. Ann.}, 333(3):569--631, 2005.

\bibitem{JPR11}
P.~Jahnke, T.~Peternell, and I.~Radloff.
\newblock Threefolds with big and nef anticanonical bundles {II}.
\newblock {\em Cent. Eur. J. Math.}, 9(3):449--488, 2011.

\bibitem{Kol89}
J.~Koll{\'a}r.
\newblock Flops.
\newblock {\em Nagoya Math. J.}, 113:15--36, 1989.

\bibitem{KMBook}
J.~Koll{\'a}r and S.~Mori.
\newblock {\em Birational geometry of algebraic varieties}, volume 134 of {\em
  Cambridge Tracts in Mathematics}.
\newblock Cambridge University Press, Cambridge, 1998.
\newblock With the collaboration of C. H. Clemens and A. Corti, Translated from
  the 1998 Japanese original.

\bibitem{MM81}
S.~Mori and S.~Mukai.
\newblock Classification of {F}ano {$3$}-folds with {$B_{2}\geq 2$}.
\newblock {\em Manuscripta Math.}, 36(2):147--162, 1981/82.

\bibitem{MM86}
S.~Mori and S.~Mukai.
\newblock Classification of {F}ano {$3$}-folds with {$B_{2}\geq 2$}. {I}.
\newblock In {\em Algebraic and topological theories ({K}inosaki, 1984)}, pages
  496--545. Kinokuniya, Tokyo, 1986.

\bibitem{MM03}
S.~Mori and S.~Mukai.
\newblock Erratum: ``{C}lassification of {F}ano 3-folds with {$B_2\geq 2$}''
  [{M}anuscripta {M}ath. {\bf 36} (1981/82), no. 2, 147--162; {MR}0641971
  (83f:14032)].
\newblock {\em Manuscripta Math.}, 110(3):407, 2003.

\bibitem{Morr84}
D.~R. Morrison.
\newblock On {$K3$} surfaces with large {P}icard number.
\newblock {\em Invent. Math.}, 75(1):105--121, 1984.

\bibitem{Muk95}
S.~Mukai.
\newblock New developments in the theory of {F}ano threefolds: vector bundle
  method and moduli problems [translation of {S}\=ugaku {\bf 47} (1995), no.\
  2, 125--144; {MR}1364825 (96m:14059)].
\newblock {\em Sugaku Expositions}, 15(2):125--150, 2002.
\newblock Sugaku expositions.

\bibitem{Rei88}
I.~Reider.
\newblock Vector bundles of rank {$2$} and linear systems on algebraic
  surfaces.
\newblock {\em Ann. of Math. (2)}, 127(2):309--316, 1988.

\bibitem{SD74}
B.~Saint-Donat.
\newblock Projective models of {$K\text{-}3$} surfaces.
\newblock {\em Amer. J. Math.}, 96:602--639, 1974.

\bibitem{Fanobook}
I.~R. Shafarevich, editor.
\newblock {\em Algebraic geometry. {V}}, volume~47 of {\em Encyclopaedia of
  Mathematical Sciences}.
\newblock Springer-Verlag, Berlin, 1999.
\newblock Fano varieties, A translation of {{\i}t Algebraic geometry. 5}
  (Russian), Ross. Akad. Nauk, Vseross. Inst. Nauchn. i Tekhn. Inform., Moscow,
  Translation edited by A. N. Parshin and I. R. Shafarevich.

\bibitem{Sho79a}
V.~V. Shokurov.
\newblock The existence of a line on {F}ano varieties.
\newblock {\em Izv. Akad. Nauk SSSR Ser. Mat.}, 43(4):922--964, 968, 1979.
\newblock English transl. in Math. USSR-Izv. 15 (1979), no. 1, 173–209
  (1980).

\bibitem{Sho79b}
V.~V. Shokurov.
\newblock Smoothness of a general anticanonical divisor on a {F}ano variety.
\newblock {\em Izv. Akad. Nauk SSSR Ser. Mat.}, 43(2):430--441, 1979.
\newblock English transl. in Math. USSR-Izv. 14 (1980), no. 2, 395–405.

\bibitem{Take89}
K.~Takeuchi.
\newblock Some birational maps of {F}ano {$3$}-folds.
\newblock {\em Compositio Math.}, 71(3):265--283, 1989.

\bibitem{Take09}
K.~Takeuchi.
\newblock Weak {F}ano threefolds with del {P}ezzo fibration.
\newblock {\em arXiv preprint arXiv:0910.2188}, 2009.

\bibitem{Vol01}
V.~Vologodsky.
\newblock On birational morphisms between pencils of del {P}ezzo surfaces.
\newblock {\em Proc. Amer. Math. Soc.}, 129(8):2227--2234 (electronic), 2001.

\end{thebibliography}

\end{document}